\documentclass[a4paper,american,reqno]{amsart}

\pdfoutput=1

\newif\ifPreprint \Preprinttrue
\newif\ifSubmission \Submissionfalse

\usepackage{babel}
\usepackage[utf8]{inputenc}
\usepackage[T1]{fontenc}
\usepackage{todonotes}
\usepackage{microtype}
\usepackage{booktabs}
\usepackage[binary-units=true]{siunitx}
\usepackage{notationTRR154}
\usepackage[linesnumbered,ruled,vlined]{algorithm2e}
\usepackage{pgfplots}
\usepackage{pgfplotstable}
\usepackage{tikz}
\usepackage{xspace}
\usepackage[colorlinks,
            citecolor=blue,
            urlcolor=blue,
            linkcolor=blue]{hyperref}

\newtheorem{lemma}{Lemma}
\newtheorem{theorem}{Theorem}
\theoremstyle{definition}
\newtheorem{definition}{Definition}

\makeatletter
\patchcmd{\@settitle}{\uppercasenonmath\@title}{\scshape\large}{}{}
\patchcmd{\@setauthors}{\MakeUppercase}{\scshape\normalsize}{}{}
\makeatother

\usetikzlibrary{external}
\usetikzlibrary{arrows}
\usetikzlibrary{decorations.pathmorphing}

\vfuzz \hfuzz
\definecolor{mycol1}{RGB}{55,126,184}
\definecolor{mycol2}{RGB}{228,26,28}
\definecolor{mycol3}{RGB}{77,175,74}
\definecolor{mycol4}{RGB}{0,0,0}
\definecolor{mycol5}{RGB}{255,127,0}
\definecolor{mycol6}{RGB}{152,78,163}
\definecolor{mycol7}{RGB}{255,255,51}

\renewcommand{\norm}[1]{\left\lVert#1\right\rVert}

\begin{document}

\title[Model and Discretization Error Adaptivity within Gas
Optimization]{Model and Discretization Error Adaptivity\\within
  Stationary Gas Transport Optimization}
\author[V. Mehrmann, M. Schmidt, J. J. Stolwijk]
{Volker Mehrmann$^{1}$,
  Martin Schmidt$^{2,3}$,
  Jeroen J.\ Stolwijk$^{1}$}
\address{$^1$Institut für Mathematik, MA 4-5, TU Berlin,
  Straße des 17. Juni 136, 10623 Berlin, Germany
  (\{mehrmann,stolwijk\}@math.tu-berlin.de);
  $^2$Friedrich-Alexander-Universität Erlangen-Nürnberg,
  Discrete Optimization,
  Cauerstr.~11,
  91058~Erlangen,
  Germany;
  $^3$Energie Campus Nürnberg,
  Fürther Str.~250,
  90429 Nürnberg,
  Germany}

\date{\today}

\begin{abstract}
  The minimization of operation costs for natural gas transport networks
is studied.
Based on a recently developed model hierarchy ranging from detailed
models of instationary partial differential equations with
temperature dependence to highly simplified algebraic equations,
modeling and discretization error estimates are presented to control the
overall error in an optimization method for stationary and isothermal
gas flows.
The error control is realized by switching to more detailed models or
finer discretizations if necessary to guarantee that a prescribed
model and discretization error tolerance is satisfied in the end.
We prove convergence of the adaptively controlled optimization method
and illustrate the new approach with numerical examples.


\end{abstract}


\maketitle

\begin{center}
  \emph{Dedicated to Hans Georg Bock on the occasion of his 70th
    birthday.}
  \vspace*{1.5em}
\end{center}

\section{Introduction}
\label{sec:introduction}

In this paper we discuss the minimization of operation costs for
natural gas transport networks based on  a model
hierarchy, see \cite{DomHLT17,Hante_et_al:2017},
which ranges from detailed models based
on instationary partial differential equations with temperature
dependence to highly simplified algebraic equations.
The detailed models are necessary to achieve a good understanding of
the system state, but in many practical optimization applications
only the stationary algebraic equations---or even further
simplifications like piecewise linearizations as in
\cite{MR3587635,MR3156585,MR3629160}---are used in order to reduce
the high computational effort of evaluating the state of the system with
the more sophisticated models.
However, it is then unclear how good the true state is approximated by
these simplified models and error bounds are typically not available
in this context; see the chapter \cite{Koch_et_al:ch11:2015} in
 \cite{Koch_et_al:2015} for a more detailed
discussion of this issue.
Recently, in \cite{StoM17}, a detailed error and perturbation
analysis has been developed for several components in the model
hierarchy and it has been shown how the more detailed model components
can be used to estimate the error obtained in the simplified models.

Here, we use these error estimates from the model hierarchy
together with classical error estimate grid adaptation techniques for
the space discretization within an optimization method to control the
error adaptively by switching to more detailed models or finer
discretizations if necessary.
Moreover, our adaptive method also allows to locally switch back to
coarser models or to coarser discretizations if they are sufficiently
accurate with respect to the local flow situation.
Our new approach can, in general, be used for the entire model
hierarchy by also using space-time grid adaptation.
However, to keep things simple and to illustrate the functionality of
the new adaptive approach, we will use three stationary
isothermal models from the hierarchy in~\cite{DomHLT17}.

Using adaptive techniques to achieve a trade-off between computational
efficiency and accuracy by using adaptive discretization methods in
the context of optimization and optimal control problems is an
important research topic, in particular in the context of real-time
optimal control of constrained dynamical systems, see,
\eg \cite{Bock_et_al:2007,Diehl_et_al:2003,Diehl_et_al:2005,Nagy_et_al:2002},
or in the context of optimal control of problems constrained by
partial differential equations; see,
\eg
\cite{Becker_et_al:2000,Kroener_et_al:2011,Leykekhman_et_al:2016,Liu_et_al:2015}.
We extend these ideas and combine adaptive grid refinement and model
selection in a model hierarchy in the context of nonlinear
optimization problems. We also theoretically analyze the new
algorithm. First promising numerical results for such an approach were
presented in \cite{Schmidt_et_al:2015a,Schmidt_et_al:2016}.

The paper is structured as follows.
The models used in this paper are described in
Sect.~\ref{sec:problem-description} together with a simple first-order
Euler method for the space discretization.
In Sect.~\ref{sec:error-estimators} we introduce model and
discretization error estimators, which are used in
Sect.~\ref{sec:algorithm} to derive an adaptive model and
discretization control algorithm for the nonlinear optimization of gas
transport networks that, in the end, delivers solutions that satisfy
prescribed error tolerances.
Numerical results are presented in Sect.~\ref{sec:comp-results} and the
paper concludes in Sect.~\ref{sec:conclusion}.


\section{Problem Description, Modeling Hierarchy, and Discretizations}
\label{sec:problem-description}

In this section we introduce the problem of operation cost
minimization for natural gas transport networks.
We present our overall model of a gas transport network
involving continuous \nonlinear models describing a stationary flow
for all the considered network elements.
Since the majority of the elements are pipes,
our focus lies on the precise and physically accurate modeling of
these pipes.
The typical models for the pipe flow are \nonlinear instationary partial
differential equations (PDEs) on a graph and their appropriate space-time
discretizations.
To address the fact that the behavior of the flow and the accuracy of
the model may vary significantly in different regions of the network,
we discuss a small part of the complete model hierarchy of instationary
models, see~\cite{DomHLT17}, where the lower level models in the
hierarchy are simplifications of the higher level models.
Which model is most appropriate to obtain a computationally tractable,
adequately accurate, and finite-dimensional approximation depends on the
task that needs to be performed with the model.

Our modeling approach is based on the following physical assumptions.
First, we only consider a stationary gas flow, \ie
we neglect all time effects of gas dynamics, so that we have
ordinary differential equations (ODEs) in space instead of systems of
PDEs on a graph.
Second, we assume an isothermal regime, \ie we neglect all effects
arising from changes in the gas temperature.

These assumptions are taken carefully such that we still obtain
physically meaningful solutions and such that we are still able to
derive and analyze an adaptive model and discretization control
algorithm---without unnecessarily overloading the models with all
technical details of the application that may distract us from the main
mathematical ideas.

\subsection{The Network}
\label{sec:problem-description:network}

We model the gas transport network by a directed and connected
graph~$\graph = (\nodes, \arcs)$.
The node set is made up of entry nodes~$\entries$, where gas is
supplied, of exit nodes~$\exits$, where gas is discharged, and of
inner nodes~$\innodes$,
\ie we have $\nodes = \entries \cup \exits \cup \innodes$.
The set of arcs in our models comprises pipes~$\arcsPipes$ and compressor
machines~$\arcsCM$, \ie we have $\arcs = \arcsPipes \cup \arcsCM$.

Real-world gas transport networks contain
many other element types like (control) valves, short cuts, or
resistors. For detailed information on modeling these devices, \cf
\cite{Koch_et_al:ch02:2015} in general or
\cite{Schmidt_et_al:2015a,Schmidt_et_al:2016} for a focus on
\nonlinear programming (NLP) type models.
However, we restrict ourselves to models with pipes and compressors in order to
streamline the presentation of our basic ideas and methods, and to
show in a  prototypical way  that our approach of space discretization
and model adaptivity leads to major accuracy and efficiency
improvements.

As basic quantities we introduce gas pressure
variables~$\pressureAtNode$ at all nodes~$\node \in \nodes$ and mass
flow variables~$\mFlowAtArc$ at all arcs~$\arc \in \arcs$ of the
network.
Both types of variables are bounded due to technical constraints on
the pipes, \ie
\begin{subequations}
  \label{eq:variable-bounds}
  \begin{align}
    \pressureAtNode & \in [\lbPressAtNode, \ubPressAtNode]
                      \qConstraintFor{\node \in \nodes},
    \\
    \mFlowAtArc & \in [\lbMassflowAtArc, \ubMassflowAtArc]
                  \qConstraintFor{\arc \in \arcs}.
  \end{align}
\end{subequations}
All other required quantities are introduced where they are used
first.

\subsection{Nodes}
\label{sec:problem-description:nodes}

In stationary gas network models, the nodes~$\node \in \nodes$ are
modeled by a mass balance equation, \ie we have the constraint
\begin{equation}
  \label{eq:node-model:mass-balance}
  \sum_{\arc \in \inArcs{\node}} \mFlowAtArc
  - \sum_{\arc \in \outArcs{\node}} \mFlowAtArc
  = \mFlowAtNode
  \qConstraintFor{\node \in \nodes},
\end{equation}
where for ingoing arcs we use the notation
\begin{equation*}
  \inArcs{\node} \define \condset{\arc \in \arcs}{\mbox{\rm there
      exists}\, \otherNode \in \nodes \text{ and } \arc = (\otherNode,
    \node)}
\end{equation*}
and for outgoing arcs
\begin{equation*}
  \outArcs{\node} \define \condset{\arc \in \arcs}{\mbox{\rm there
      exists}\, \otherNode \in \nodes \text{ and } \arc = (\node,
    \otherNode)}.
\end{equation*}
Moreover, $\mFlowAtNode$ models the supplied or discharged mass flow at
the corresponding node, \ie we have
\begin{equation*}
  \mFlowAtNode
  \begin{cases}
    \geq 0 & \text{for all } \node \in \exits,
    \\
    \leq 0 & \text{for all } \node \in \entries,
    \\
    = 0 & \text{for all } \node \in \innodes.
  \end{cases}
\end{equation*}

\subsection{Pipes}
\label{sec:problem-description:pipes}

Isothermal gas flow through cylindrical pipes is described by the
Euler equations for compressible fluids,
\begin{subequations}
  \label{eq:euler-iso1-p-v-version}
  \begin{align}
    \frac{\partial\density}{\partial t} +
    \frac{1}{\area}\frac{\partial\mFlow}{\partial \spatVar}
    &= \num{0},
      \label{eq:euler-iso1-p-v-version-cont}\\
    \frac{1}{\area}\frac{\partial\mFlow}{\partial t} +
    \frac{\partial \pressure}{\partial \spatVar} +
    \frac{1}{\area}\frac{\partial(\mFlow \velocity)}{\partial \spatVar}
    &= - \lambda(\mFlow)\frac{\abs{\velocity}\velocity}{2\diameter}\density
      -\gravity\density \slope,
      \label{eq:euler-iso1-p-v-version-mom}
  \end{align}
\end{subequations}
\cf, \eg \cite{Feistauer:1993,Lurie:2008} for a detailed discussion.
Here and in what follows, $\density$ is the gas density,
$\velocity$ is its velocity, $\lambda = \lambda(\mFlow)$
is the friction term, $\area$ denotes the cross-sectional area of the
pipe, $\slope$ is its slope, and $\diameter$ is the diameter of the
pipe.
Furthermore, $\gravity$ is the acceleration due to gravity, $t$ is the temporal
coordinate, and $\spatVar \in [0, \length]$ is the spatial coordinate with
$\length$ being the length of the pipe.
Equation~\eqref{eq:euler-iso1-p-v-version-cont} is called the
continuity equation and \eqref{eq:euler-iso1-p-v-version-mom}
the momentum equation.
Since we only consider the stationary case, all partial derivatives
\wrt time vanish and we obtain the simplified stationary model
\begin{subequations}
  \label{eq:stationary-euler-iso1-p-v-version}
  \begin{align}
    \frac{1}{\area}\frac{\partial\mFlow}{\partial \spatVar}
    &= \num{0},
      \label{eq:stationary-euler-iso1-p-v-version-cont}\\
    \frac{\partial \pressure}{\partial \spatVar} +
    \frac{1}{\area}\frac{\partial(\mFlow \velocity)}{\partial \spatVar}
    &= -\lambda(\mFlow)
      \frac{\abs{\velocity}\velocity}{2\diameter}\density
      -\gravity\density \slope.
      \label{eq:stationary-euler-iso1-p-v-version-mom}
  \end{align}
\end{subequations}
Thus, the continuity equation in its stationary variant simply states
that the mass flow along the pipe is constant, \ie $\mFlow(\spatVar) \equiv
\mFlow = \const$ for all $\spatVar \in [0, \length]$.

To simplify the stationary momentum
equation~\eqref{eq:stationary-euler-iso1-p-v-version-mom}, we consider
two more model equations. First, the equation of state
\begin{equation*}
  \press = \density \speedOfSound^2
  \quad \text{with} \quad
  \speedOfSound = \sqrt{\specificGasConstant \temperature \realGasFactor},
\end{equation*}
where $\speedOfSound$ is the speed of sound, $\specificGasConstant$ is
the specific gas constant, and $\realGasFactor$ is the compressibility
factor. The second model is the relation of gas mass flow, density, and velocity
given by
\begin{equation*}
  \mFlow = \area \density \velocity.
\end{equation*}
Substituting both these models
into~\eqref{eq:stationary-euler-iso1-p-v-version-mom}, we obtain
\begin{equation}
  \tag{M$_1$}
  \label{eq:stationary-momentum-equ}
  \frac{\partial \pressure}{\partial \spatVar}
  \left(1 - \frac{\mFlow^2}{\area^2}\frac{\speedOfSound^2}{\pressure^2} \right)
  = -\frac{\lambda \speedOfSound^2}{2 \area^2 \diameter \pressure}
  \abs{\mFlow}\mFlow
  -\frac{\gravity \slope}{\speedOfSound^2} \pressure,
\end{equation}
\ie the stationary momentum equation written in dependence of
the gas pressure $\pressure = \pressure(x)$, $\spatVar \in [0,
\length]$, and the mass flow~$\mFlow$.

A simplified version of the latter equation can be obtained by
ignoring the ram pressure term
\begin{equation*}
  \frac{1}{\area}\frac{\partial(\mFlow \velocity)}{\partial \spatVar},
\end{equation*}
in \eqref{eq:stationary-euler-iso1-p-v-version-mom}, \ie the total
pressure exerted on the gas by the pipe wall, or, equivalently,
the term
\begin{equation}\label{ram_pressure}
  -\frac{\mFlow^2}{\area^2}\frac{\speedOfSound^2}{\pressure^2}
  \frac{\partial \pressure}{\partial \spatVar}
\end{equation}
in \eqref{eq:stationary-momentum-equ}.
For a discussion of this simplification step,
see~\cite{WilkinsonEtAl:1964}.
Neglecting the ram pressure term~\eqref{ram_pressure} yields
\begin{equation}
  \tag{M$_2$}
  \label{eq:stationary-momentum-equ-wo-ram-pressure}
  \frac{\partial \pressure}{\partial \spatVar}
  = -\frac{\lambda \speedOfSound^2}{2 \area^2 \diameter \pressure}
  \abs{\mFlow}\mFlow
  -
  \frac{\gravity \slope}{\speedOfSound^2} \pressure.
\end{equation}
Finally, one may also neglect gravitational forces, \ie set the term
$\gravity \slope \pressure / \speedOfSound^2$ to~$0$ and obtain
\begin{equation}
  \tag{M$_3$}
  \label{eq:stationary-momentum-equ-wo-ram-pressure-and-heights}
  \frac{\partial \pressure}{\partial \spatVar}
  = -\frac{\lambda \speedOfSound^2}{2 \area^2 \diameter \pressure}
  \abs{\mFlow}\mFlow.
\end{equation}

Analytical solutions for the models~\eqref{eq:stationary-momentum-equ}--%
\eqref{eq:stationary-momentum-equ-wo-ram-pressure-and-heights} are
only rarely known; \cf, \eg
\cite{Gugat_et_al:2015,Gugat_et_al:2016,Schmidt_et_al:2015a}.
Thus, in order to obtain finite-dimensional nonlinear optimization
models, we discretize these differential equations in space.
Applying, \eg the implicit Euler method we obtain
\begin{align}
  \tag{D$_1$}
  \label{eq:stationary-momentum-equ-discretized}
  \frac{\pressure_k - \pressure_{k-1}}{h}
  \left(1 - \frac{\mFlow^2}{\area^2}\frac{\speedOfSound^2}{\pressure_k^2} \right)
  & = -
    \frac{\lambda \speedOfSound^2}{2 \area^2 \diameter \pressure_k} \abs{\mFlow}\mFlow
    -
    \frac{\gravity \slope}{\speedOfSound^2} \pressure_k,
  & k = 1, \dotsc, n,\\
  \tag{D$_2$}
  \label{eq:stationary-momentum-equ-wo-ram-pressure-discretized}
  \frac{\pressure_k - \pressure_{k-1}}{h}
  & = -
    \frac{\lambda \speedOfSound^2}{2 \area^2 \diameter \pressure_k} \abs{\mFlow}\mFlow
    -
    \frac{\gravity \slope}{\speedOfSound^2} \pressure_k,
  & k = 1, \dotsc, n,\\
  \tag{D$_3$}
  \label{eq:stationary-momentum-equ-wo-ram-pressure-and-heights-discretized}
  \frac{\pressure_k - \pressure_{k-1}}{h}
  & = -
    \frac{\lambda \speedOfSound^2}{2 \area^2 \diameter \pressure_k}
    \abs{\mFlow}\mFlow,
  & k = 1, \dotsc, n,
\end{align}
where $p_k = p(\spatVar_k)$ and $\Gamma = \set{\spatVar_0, \spatVar_1,
  \dotsc, \spatVar_n}$ is an equidistant spatial discretization of the
pipe with constant stepsize $h = \spatVar_k - \spatVar_{k-1}$ and $x_0
= 0, x_n = L$.
Of course, one could also apply a higher-order Runge--Kutta method,
which would allow a larger stepsize and would thus reduce the
computational cost.

These discretizations extend the model
hierarchy \eqref{eq:stationary-momentum-equ}--%
\eqref{eq:stationary-momentum-equ-wo-ram-pressure-and-heights}
for the Euler equations by infinitely many models that are
parameterized by the discretization stepsize~$h$ applied in
\eqref{eq:stationary-momentum-equ-discretized}--%
\eqref{eq:stationary-momentum-equ-wo-ram-pressure-and-heights-discretized}.
In summary, we obtain the pipe model hierarchy of stationary
Euler equations depicted in Fig.~\ref{fig:ModelHierarchy}.
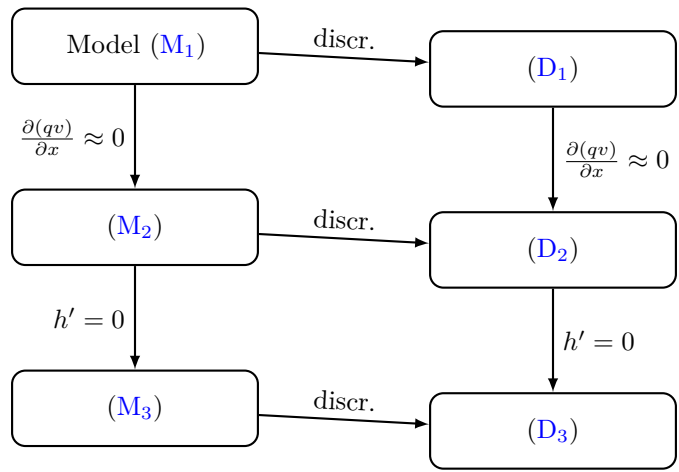
\begin{figure}
  \centering
  \begin{tikzpicture}
  [model/.style = {rectangle, draw, thick, rounded corners=5pt,minimum height=1cm,text centered,draw=black,text width=3cm},
  edge from parent/.style={draw,-latex,thick},
  level distance=2.4cm]
  %

  \node (M1) [model] at (0,0) {Model \eqref{eq:stationary-momentum-equ}}
  child {node[model] (M2) {\eqref{eq:stationary-momentum-equ-wo-ram-pressure}}
    child {node[model] (M3) {\eqref{eq:stationary-momentum-equ-wo-ram-pressure-and-heights}}
      edge from parent node[left] {$\slope=0$}}
    edge from parent node[left] {$\frac{\partial(\mFlow
        \velocity)}{\partial \spatVar} \approx 0$}
  }
  ;

  \node (D1) [model] at (5.5,-0.3) {\eqref{eq:stationary-momentum-equ-discretized}}
  child {node[model] (D2) {\eqref{eq:stationary-momentum-equ-wo-ram-pressure-discretized}}
    child {node[model] (D3) {\eqref{eq:stationary-momentum-equ-wo-ram-pressure-and-heights-discretized}}
      edge from parent node[right] {$\slope=0$}}
    edge from parent node[right] {$\frac{\partial(\mFlow
        \velocity)}{\partial \spatVar} \approx 0$}
  }
  ;

  \path[draw,-latex,thick] (M1) -> node[above,sloped] {discr.} (D1);
  \path[draw,-latex,thick] (M2) -> node[above,sloped] {discr.} (D2);
  \path[draw,-latex,thick] (M3) -> node[above,sloped] {discr.} (D3);
\end{tikzpicture}

  \caption{Pipe model hierarchy based on the Euler equations.
    The space continuous models are positioned in the left column and
    their space discretized counterparts are positioned in the right
    column.}
  \label{fig:ModelHierarchy}
\end{figure}

\subsection{Compressors}
\label{sec:problem-description:compr-model}

Compressor machines~$\arc = (\node, \otherNode) \in \arcsCM$
increase the inflow gas pressure to a higher outflow pressure, \ie
they can be described in a simplified way by
\begin{equation}
  \label{eq:compr-model}
  \pressureAtOtherNode = \pressureAtNode + \pressureDiffAtCM,
  \quad
  \pressureDiffAtCM \in [0, \ubPressureDiffAtCM]
  \quad \text{ for all } \arc \in \arcsCM.
\end{equation}
Moreover, for simplicity, we assume that we are given cost
coefficients~$\omega_\arc \geq 0$ for every compressor~$\arc \in
\arcsCM$ that converts pressure increase to compression cost.
Of course, this is an extremely coarse approximation of a compressor
machine. An alternative would be to use a simple  input-output
surrogate model obtained from a realization or system identification
of an input-output transfer function; \cf, \eg~\cite{Bro15}.
However, our focus is on an accurate modeling of the gas flow in pipes
and on deriving an adaptive model and discretization control algorithm.
Model~\eqref{eq:compr-model} allows for setting up a reasonable
objective function for our NLPs and is thus appropriate in this work.
For more details, see
\cite{Rose_et_al:2016,Schmidt_et_al:2015a,Schmidt_et_al:2016} or
\cite{Koch_et_al:ch02:2015}.

\subsection{The Optimization Problem}
\label{sec:entire-model}

We will use the adaptive model and discretization control algorithm in
the context of the following \nonlinear ODE-constrained optimization
problem
\begin{subequations}
  \label{eq:entire-model-infinite-dim}
  \begin{align}
    \min \quad
    & \sum_{\arc \in \arcsCM} \omega_\arc \pressureDiffAtCM
    \\
    \st \quad
    & \text{variable bounds } \eqref{eq:variable-bounds},
      \label{eq:entire-model-infinite-dim:var-bounds}\\
    & \text{mass balance } \eqref{eq:node-model:mass-balance},
    \\
    & \text{compressor model } \eqref{eq:compr-model}
      \text{ for all } \arc \in \arcsCM,
    \label{eq:entire-model-infinite-dim:compressor} \\
    & \text{pipe model }
      \text{\eqref{eq:stationary-momentum-equ}}
      \text{ for all } \arc \in \arcsPipes,
      \label{eq:entire-model-infinite-dim-pipe}
  \end{align}
\end{subequations}
where our objective function models the cost for the compressor activity
that is constrained by an infinite-dimensional description of the gas
flow in pipes.
Problem~\eqref{eq:entire-model-infinite-dim}
is a classical nonlinear optimal control problem.
A typical approach to solve such problems in practice is the
first-discretize-then-optimize paradigm; \cf, \eg~\cite{Bie10}.
In this setting, one replaces the ODE constraints by finite sets of
\nonlinear constraints that arise, \eg from implicit Euler
discretizations like \eqref{eq:stationary-momentum-equ-discretized}
for \eqref{eq:stationary-momentum-equ}.
Moreover, practical experience suggests that for the evaluation of the
constraints, it is often not required to apply the most accurate model like
\eqref{eq:stationary-momentum-equ-discretized} with a small stepsize
for every pipe in the network.
Instead, in many situations it is sufficient  to use simplified models like
\eqref{eq:stationary-momentum-equ-wo-ram-pressure-discretized} and
\eqref{eq:stationary-momentum-equ-wo-ram-pressure-and-heights-discretized}
with a coarse grid,
which then typically yields fast execution times for the evaluation of
the constraint functions.

To this end, we define discretized problem variants of
Problem~\eqref{eq:entire-model-infinite-dim} by specifying the
model level~$\ell_\arc \in \set{1, 2, 3}$
for every arc~$\arc \in \arcsPipes$ (\ie the discretized model
\eqref{eq:stationary-momentum-equ-discretized},
\eqref{eq:stationary-momentum-equ-wo-ram-pressure-discretized}, or
\eqref{eq:stationary-momentum-equ-wo-ram-pressure-and-heights-discretized},
respectively)
together with a stepsize~$h_\arc$.
This yields the family of finite-dimensional NLPs
\begin{subequations}
  \label{eq:entire-model-discr}
  \begin{align}
    \min \quad
    & \sum_{\arc \in \arcsCM} \omega_\arc \pressureDiffAtCM
    \\
    \st \quad
    & \text{variable bounds } \eqref{eq:variable-bounds},
      \label{eq:entire-model-discr:var-bounds}\\
    & \text{mass balance } \eqref{eq:node-model:mass-balance},
    \\
    & \text{compressor model } \eqref{eq:compr-model}
      \text{ for all } \arc \in \arcsCM,
      \label{eq:entire-model-discr:compressor} \\
    & \text{pipe model (\Dla) with stepsize } h_\arc
      \text{ for all } \arc \in \arcsPipes.
      \label{eq:entire-model-discr:pipe}
  \end{align}
\end{subequations}
Note that the
constraints~\eqref{eq:entire-model-infinite-dim:var-bounds}--%
\eqref{eq:entire-model-infinite-dim:compressor} in the infinite-dimensional
problem are exactly the same as
constraints~\eqref{eq:entire-model-discr:var-bounds}--%
\eqref{eq:entire-model-discr:compressor} in the family of discretized problems.


\section{Error Estimators}
\label{sec:error-estimators}

In this section we introduce a \firstorder
estimate for the error between the most detailed infinite-dimensional
and an arbitrary space-discretized model.
This error estimator is obtained as the sum of a
discretization and a model error estimator.
Since we consider the stationary case, mass flows in pipes are
constant in the spatial dimension.
This is why we base our error estimators on the differences of the
pressures $\pressure(\spatVar)$ for different models and
discretizations.

Suppose that for a given pipe $\arc \in \arcsPipes$,
the model level~$\la \in \set{1,2,3}$ with discretization
stepsize~$h_\arc$ is currently used for the computations.
The overall solution of the optimization problem for the
entire network, also including pressure
increases in compressors \etc, is denoted by~$\sol$
and contains the discretized pressure distributions of
the separate pipes~$\arc \in \arcsPipes$, which we denote
by $\pressure^{\la}(\spatVar_k;h_\arc)$ with
discretization grid $\Gamma_1 = \set{\spatVar_k}_{k=0}^{L_\arc/h_\arc}$
obtained by using the stepsize~$h_\arc$.
We now compute an estimate for the error between the solution of the
currently used model~$(\Dla)$ and the solution of the reference
model~\eqref{eq:stationary-momentum-equ}.
Let the solution of model~\eqref{eq:stationary-momentum-equ} for pipe
$\arc \in \arcsPipes$ be denoted by $\hat{\pressure}(\spatVar)$
with $\spatVar \in [0,L_\arc]$.

\begin{figure}
  \centering
  \begin{tikzpicture}
  %

  \draw[thick,color=black!50] (0,2.5) node{} -- (0,-0.5) node{};
  \draw[thick,color=black!50] (4,2.5) node{} -- (4,-0.5) node{};
  \draw[thick,color=black!50] (8,2.5) node{} -- (8,-0.5) node{};

  \fill (0,2) circle[radius=2pt];
  \fill (1,2) circle[radius=2pt];
  \fill (2,2) circle[radius=2pt];
  \fill (3,2) circle[radius=2pt];
  \fill (4,2) circle[radius=2pt];
  \fill (5,2) circle[radius=2pt];
  \fill (6,2) circle[radius=2pt];
  \fill (7,2) circle[radius=2pt];
  \fill (8,2) circle[radius=2pt];

  \node at (9.5,1.97) {$\Gamma_1 = \set{\spatVar_k}_{k=0}^{L_\arc/h_\arc}$};

  \fill (0,1) circle[radius=2pt];
  \fill (2,1) circle[radius=2pt];
  \fill (4,1) circle[radius=2pt];
  \fill (6,1) circle[radius=2pt];
  \fill (8,1) circle[radius=2pt];

  \node at (9.66,0.97) {$\Gamma_2 = \set{\spatVar_s}_{s=0}^{L_\arc/(2h_\arc)}$};

  \fill (0,0) circle[radius=2pt];
  \fill (4,0) circle[radius=2pt];
  \fill (8,0) circle[radius=2pt];

  \node at (9.66,-0.03) {$\Gamma_3 = \set{\spatVar_r}_{r=0}^{L_\arc/(4h_\arc)}$};

  \node (a) at (-0.1,2) {};
  \node (b) at (1.1,2) {};

  \draw [
    thick,
    decoration={
        brace,
        mirror,
        raise=0.1cm
    },
    decorate
  ] (a) -- (b)
  node [pos=0.5,anchor=north,yshift=-0.15cm] {$h_\arc$};

  \node (c) at (-0.1,1) {};
  \node (d) at (2.1,1) {};

  \draw [
    thick,
    decoration={
        brace,
        mirror,
        raise=0.1cm
    },
    decorate
  ] (c) -- (d)
  node [pos=0.5,anchor=north,yshift=-0.15cm] {$2h_\arc$};

  \node (e) at (-0.1,0) {};
  \node (f) at (4.1,0) {};

  \draw [
    thick,
    decoration={
        brace,
        mirror,
        raise=0.1cm
    },
    decorate
  ] (e) -- (f)
  node [pos=0.5,anchor=north,yshift=-0.15cm] {$4h_\arc$};

\end{tikzpicture}

  \caption{Overview of the three considered discretization
    grids~$\Gamma_1$, $\Gamma_2$, and~$\Gamma_3$ with
    gridpoints~$\spatVar_k$, $\spatVar_s$, and $\spatVar_r$
    and stepsizes~$h_\arc$, $2h_\arc$, and $4h_\arc$, respectively.
    The vertical lines represent the evaluation grid~$\Gamma_3$
    for the error estimators in~\eqref{eq:discretizationErrorEstimator}
    and~\eqref{eq:modelErrorEstimator}.}
  \label{fig:evaluationGrid}
\end{figure}
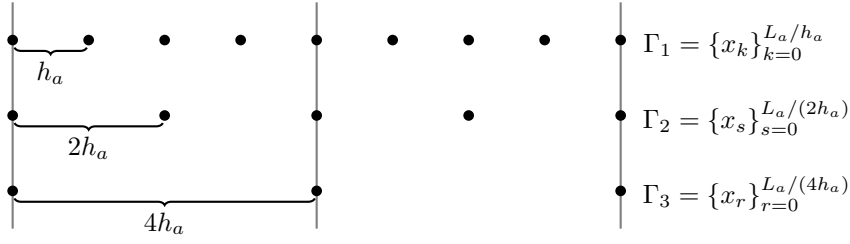

Furthermore, let the solutions of
Model~\eqref{eq:stationary-momentum-equ-discretized} with
discretization grids $\Gamma_2 = \set{\spatVar_s}_{s=0}^{L_\arc/(2h_\arc)}$
and $\Gamma_3 = \set{\spatVar_r}_{r=0}^{L_\arc/(4h_\arc)}$
using stepsizes~$2h_\arc$ and~$4h_\arc$,  be denoted by
$\pressure^1(\spatVar_s;2h_\arc)$ and $\pressure^1(\spatVar_r;4h_\arc)$,
respectively.
Due to the larger stepsize, the computation of
these two solutions is in general less expensive than computing a
solution of Model~$(\Dla)$ on the grid~$\Gamma_1$.
Since the discretization grid $\Gamma_3$ is the coarsest grid and
all computed pressure profiles can be evaluated on this grid,
$\Gamma_3$ is called the evaluation grid.
This grid is used in the definitions of the following error
estimators.
The considered discretization grids and the evaluation grid
are depicted in Fig.~\ref{fig:evaluationGrid}.

For a pipe~$\arc \in \arcsPipes$, let the 
discretization error
estimator be defined by
\begin{equation}\label{eq:discretizationErrorEstimator}
  \eta_{\td,\arc}(\sol) \define
  \norm{
    \pressure^1(\spatVar_r;2h_\arc) -
      \pressure^1(\spatVar_r;4h_\arc)
      }_\infty
\end{equation}
and let the 
model error estimator be defined by
\begin{equation}\label{eq:modelErrorEstimator}
  \eta_{\tm,\arc}(\sol) \define
  \norm{
    \pressure^1(\spatVar_r;2h_\arc) -
      \pressure^{\la}(\spatVar_r;h_\arc)
      }_\infty.
\end{equation}
Here,
\begin{equation*}
  \pressure^{\la}(\spatVar_r;h_\arc) =
  [\pressure^{\la}(\spatVar_0;h_\arc), \dotsc,
  \pressure^{\la}(\spatVar_n;h_\arc)]^\top,
  \quad
  n = L_\arc/(4h_\arc),
\end{equation*}
denotes the solution of Model~$(\Dla)$ computed with stepsize~$h_\arc$
that is evaluated at the gridpoints~$\spatVar_r$, \ie on the grid~$\Gamma_3$.
If $\la = 1$, \ie if the considered solution already corresponds to the
most accurate model, then we set the
model error to zero, \ie $\eta_{\tm,\arc}(\sol) = 0$.
Furthermore, let the overall error estimator~$\eta_\arc(\sol)$ for a pipe~$\arc
\in \arcsPipes$ be defined to be a \firstorder upper bound for the
maximum error between the solutions of
models~\eqref{eq:stationary-momentum-equ} and~$(\Dla)$ at
gridpoints~$\spatVar_r$ with stepsize~$4h_\arc$.
Thus, we have
\begin{equation}
  \label{eq:firstOrderApprox}
  \begin{split}
    & \norm{
      \hat{\pressure}(\spatVar_r) -
        \pressure^{\la}(\spatVar_r;h_\arc)
        }_\infty
    \\
    \leq \
    & \norm{
      \hat{\pressure}(\spatVar_r) -
        \pressure^1(\spatVar_r;2h_\arc)
        }_\infty
      + \norm{
      \pressure^1(\spatVar_r;2h_\arc) -
        \pressure^{\la}(\spatVar_r;h_\arc)
        }_\infty
    \\
    \doteq \
    & \eta_{\td,\arc}(\sol) + \eta_{\tm,\arc}(\sol)
    \enifed \eta_\arc(\sol),
  \end{split}
\end{equation}
where $\doteq$ denotes a \firstorder approximation in~$h_\arc$,
see \cite[page~420]{StoB80}, and we use that the implicit
Euler method has convergence order~1.
The error estimator~$\eta_\arc(\sol)$ is the absolute counterpart
of the componentwise relative error estimator given in
\cite{StoM17}.
An overview of the considered models in this section together with
the considered stepsizes is depicted in Fig.~\ref{fig:ModelOverview}.

We close this section with a remark on the computation of the
discretization error estimator
in~\eqref{eq:discretizationErrorEstimator}.
A straightforward way is to solve
Model~\eqref{eq:stationary-momentum-equ-discretized} once with
stepsize~$2h_\arc$ and once again with stepsize~$4h_\arc$ for every
$\arc \in \arcsPipes$.
Another possibility would be to use an embedded Runge--Kutta method,
see, \eg \cite{HaiNW93}, which in general saves computational cost
due to the reduced number of function evaluations.

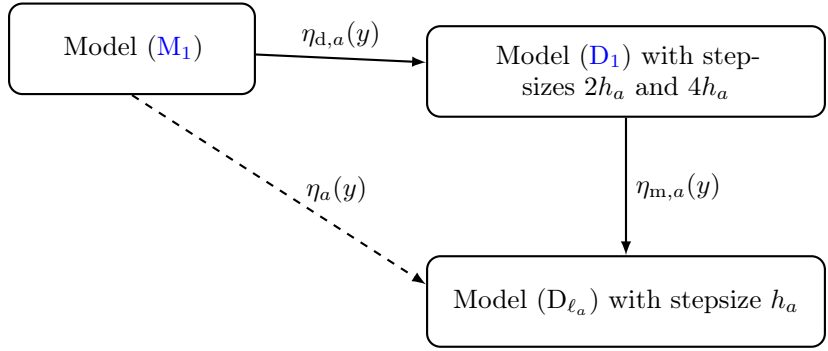
\begin{figure}
  \centering
  \begin{tikzpicture}
  \tikzstyle{model_D} = [%
  rectangle,
  rounded corners=5pt,
  minimum width=2cm,
  minimum height=1.2cm,
  text centered,
  draw=black,
  thick,
  text width=5cm]
  \tikzstyle{model_M} = [%
  rectangle,
  rounded corners=5pt,
  minimum width=2cm,
  minimum height=1.2cm,
  text centered,
  draw=black,
  thick,
  text width=3cm]
  %

  \node (M1) [model_M] at (-6.5,2.55) {Model~\eqref{eq:stationary-momentum-equ}};

  \node (D1) [model_D] at (0,2.25)
  {Model~\eqref{eq:stationary-momentum-equ-discretized} with
    stepsizes~$2h_\arc$ and~$4h_\arc$
  };

  \node (Dl) [model_D] at (0,-0.8) {Model~$(\Dla)$ with
    stepsize~$h_\arc$
    };

  \path[draw,-latex,thick] (M1) -> node[above] {$\eta_{\td,\arc}(\sol)$} (D1);
  \path[draw,-latex,thick] (D1) -> node[right] {$\eta_{\tm,\arc}(\sol)$} (Dl);
  \path[draw,-latex,thick,dashed] (M1.south) -> node[right,xshift=1.5ex] {$\eta_\arc(\sol)$} ([yshift=1.5ex]Dl.west);
\end{tikzpicture}

  \caption{Overview of the models and stepsizes used for
    the computation of the overall error estimator~$\eta_\arc(\sol)$ between
    models~\eqref{eq:stationary-momentum-equ} and~$(\Dla)$
    in~\eqref{eq:firstOrderApprox}.
    Here, for a pipe~$\arc$, $\eta_{\td,\arc}(\sol)$ is the
    discretization error estimator and $\eta_{\tm,\arc}(\sol)$ is the
    model error estimator.}
  \label{fig:ModelOverview}
\end{figure}


\section{The Grid and Model Adaptation Algorithm}
\label{sec:algorithm}

In this section we present and analyze an algorithm that adaptively switches
between the model levels in the hierarchy of Fig.~\ref{fig:ModelHierarchy}
and adapts discretization stepsizes in order to find a convenient
trade-off between physical accuracy and computational costs.
To this end, the algorithm iteratively solves NLPs and initial value
problems (IVPs).
Solutions of the latter are used to evaluate the error estimators
discussed in the last section and to decide on the model levels and
the discretization stepsizes for the next NLP.

Consider a single NLP of the sequence of NLPs that are solved during
the algorithm and assume that pipe~$\arc \in \arcsPipes$ is modeled using
model~$(\Dla)$ and stepsize~$h_\arc$.
Let the solution of this NLP be denoted by $\sol$.
According to the last section, the overall 
model and
discretization error estimator for this pipe is
given by $\eta_\arc(\sol)$ as defined in~\eqref{eq:firstOrderApprox}.
Thus, it is given by the 
error estimator between the solutions
of the most accurate model~\eqref{eq:stationary-momentum-equ} and the
current model~$(\Dla)$.

The overall goal of our method is to compute a solution of a
member of the family of discretized
problems~\eqref{eq:entire-model-discr} for which it is guaranteed that
this solution has an estimated average error per pipe \wrt\ the
reference model~\eqref{eq:stationary-momentum-equ}, that is less than
an a-priorily given tolerance~$\eps > 0$.
This leads us to the following definition:
\begin{definition}[$\eps$-feasibility]
  \label{def:epsilon-feasibility}
  Let $\eps > 0$ be given.
  We say that a solution~$\sol$ of problem~\eqref{eq:entire-model-discr}
  with discretized models~(\Dla), $\ell_\arc \in \set{1,2,3}$, and
  stepsizes~$h_\arc$ for the pipes~$\arc \in \arcsPipes$ is
  \emph{$\eps$-feasible} \wrt\ the reference
  problem~\eqref{eq:entire-model-infinite-dim} if
  \begin{equation*}
    \frac{1}{\card{\arcsPipes}}
    \sum_{\arc \in \arcsPipes}
    \eta_\arc(\sol)
    \leq
    \eps.
  \end{equation*}
\end{definition}
The remainder of this section is organized as follows.
Sect.~\ref{sec:model-adap-rules} introduces rules about how
the model levels and discretization stepsizes are modified.
The strategies for marking pipes for model or grid adaptation
are explained in Sect.~\ref{sec:marking-strategies}.
The adaptive model and discretization control algorithm
is introduced in Sect.~\ref{sec:algorithm-subsec},
together with a theorem for the finite termination
of the algorithm.
Finally, some remarks regarding the adaptive control algorithm
are given in Sect.~\ref{sec:remarks}.

\subsection{Model and Discretization Adaptation Rules}
\label{sec:model-adap-rules}

Before we present and discuss the overall adaptive model control
algorithm we have to
\begin{enumerate}
\item describe the mechanisms of switching up or down pipe
  model levels as well as that of refining and coarsening the discretization
  grids, and
\item discuss our marking strategy that determines the arcs on which
  the model or grid should be adapted.
\end{enumerate}
We start with the first issue and follow the standard PDE grid
adaptation technique; see, \eg \cite{CarH06a,CarH06b,Doe96} or
\cite{BreS07}.
The general strategy is as follows.
We switch up one level in the model hierarchy if this yields an error
reduction that is larger than $\eps$;
otherwise, we switch up to the most accurate discretized
model~\eqref{eq:stationary-momentum-equ-discretized}.
Hence, for pipe $\arc \in \arcsPipes$ we have the rule
\begin{equation}\label{level_switch_up_rule}
  \ell^\tnew_\arc =
  \begin{cases}
    \ell_\arc - 1, & \quad \text{if }
    \eta_{\tm,\arc}(\sol; \ell_\arc) - \eta_{\tm,\arc}(\sol;\ell_\arc - 1)
    > \eps, \\
    1, & \quad \text{otherwise},
  \end{cases}
\end{equation}
for switching up levels in the model hierarchy.
We apply this rule because it is possible that the effects of
neglecting the ram pressure term (which is the difference between
model levels $\ell = 1$ and $\ell = 2$) and neglecting gravitational
forces for non-horizontal pipes (which is the difference between
model levels $\ell = 2$ and $\ell = 3$)
balance each other out in the computation of the pressure profile
of model~\eqref{eq:stationary-momentum-equ-wo-ram-pressure-and-heights-discretized}.
In this case, switching from model~$(\tD_3)$ to~$(\tD_2)$
would increase the model error,
which is why we switch from~$(\tD_3)$ to~$(\tD_1)$ directly.

A discretization grid refinement or coarsening with a factor
$\gamma > 1$ is defined by taking the new stepsize as
\begin{equation}\label{h_new}
  h_\arc^{\text{new}} \define
  \begin{cases}
    h_\arc/\gamma, & \quad \text{for a grid refinement},\\
    \gamma h_\arc, & \quad \text{for a grid coarsening}.
  \end{cases}
\end{equation}
For a discretization scheme of order $\beta$ it is well-known that a
\firstorder approximation for the discretization
error in $x \in [0,L_\arc]$ is given by $e_{\td,\arc}(x) \doteq c(x)h_\arc^\beta$,
where $c(x)$ is independent of $h_\arc$;
\cf, \eg \cite{StoB80}.
From this, it follows that the new discretization error
after a grid refinement or coarsening can be
written as
\begin{equation*}
  e_{\td,\arc}^\tnew(x) \doteq (h_\arc^{\text{new}}/h_\arc)^\beta e_{\td,\arc}(x).
\end{equation*}
Since the implicit Euler method has convergence order~$\beta = 1$,
with $h_\arc^\tnew$ in~\eqref{h_new} and $\gamma = 2$,
for the new discretization error estimator after a grid
refinement or coarsening, it holds that
\begin{equation}\label{new_discr_error}
  \eta_{\td,\arc}^\tnew(\sol) \doteq
  \begin{cases}
    \eta_{\td,\arc}(\sol)/2,
    & \quad \text{for a grid refinement},\\
    2 \eta_{\td,\arc}(\sol),
    & \quad \text{for a grid coarsening}.
  \end{cases}
\end{equation}

\subsection{Marking Strategies}
\label{sec:marking-strategies}

We now describe our marking strategies, \ie how we choose which pipes
should be switched up or down in their model level and which pipes
should get a refined or coarsened grid.
Given marking strategy parameters~$\Theta_\td,
\Theta_\tm \in [0,1]$, we compute subsets $\refineArcs, \switchUpArcs \subseteq
\arcsPipes$ such that they are the minimal subsets of arcs that
satisfy
\begin{equation}
  \label{eq:marking-strategy-refine}
  \Theta_\td \sum_{\arc \in \arcsPipes}
  \eta_{\td, \arc}(\sol) \leq
  \sum_{\arc \in \refineArcs} \eta_{\td, \arc}(\sol)
\end{equation}
and
\begin{equation}
  \label{eq:marking-strategy-up}
  \Theta_\tm \sum_{\arc \in \arcsPipes^{> \eps}}
  (\eta_{\tm,\arc}(\sol;\ell_\arc) - \eta_{\tm,\arc}(\sol;\ell_\arc^\tnew))
  \leq
  \sum_{\arc \in \switchUpArcs}
  (\eta_{\tm,\arc}(\sol;\ell_\arc) -
  \eta_{\tm,\arc}(\sol;\ell_\arc^\tnew))
\end{equation}
with
\begin{equation*}
  \arcsPipes^{> \eps} \define
  \condset{\arc \in \arcsPipes}
  {\eta_{\tm,\arc}(\sol;\ell_\arc) - \eta_{\tm,\arc}(\sol;\ell_\arc^\tnew)
    > \eps},
\end{equation*}
where $\la^\tnew$ is given in \eqref{level_switch_up_rule}.
Analogously, given marking strategy parameters $\Phi_\td, \Phi_\tm \in
[0,1]$ and $\tau \geq 1$, we compute $\coarsenArcs, \switchDownArcs
\subseteq \arcsPipes$ such that they are the maximal subsets of
arcs that satisfy
\begin{equation}
  \label{eq:marking-strategy-coarsen}
  \Phi_\td \sum_{\arc \in \arcsPipes}
  \eta_{\td, \arc}(\sol) \geq
  \sum_{\arc \in \coarsenArcs} \eta_{\td, \arc}(\sol)
\end{equation}
and
\begin{equation}
  \label{eq:marking-strategy-down}
  \Phi_\tm \sum_{\arc \in \arcsPipes^{< \eps}(\tau)}
  (\eta_{\tm,\arc}(\sol;\ell_\arc^\tnew) - \eta_{\tm,\arc}(\sol;\ell_\arc))
  \geq
  \sum_{\arc \in \switchDownArcs}
  (\eta_{\tm,\arc}(\sol;\ell_\arc^\tnew) -
  \eta_{\tm,\arc}(\sol;\ell_\arc))
\end{equation}
with
\begin{equation*}
  \arcsPipes^{< \eps}(\tau) \define
  \condset{\arc \in \arcsPipes}
  {\eta_{\tm,\arc}(\ell_\arc^\tnew) - \eta_{\tm,\arc}(\ell_\arc)
    \leq \tau \eps}.
\end{equation*}
In~\eqref{eq:marking-strategy-down}, $\ell_\arc^\tnew$ is always set
to $\min\set{\ell_\arc + 1, 3}$.
For every arc~$\arc \in \refineArcs$ ($\arc \in \coarsenArcs$) we refine
(coarsen) the discretization grid
by halving (doubling) the stepsize, \ie we set 
$\gamma = 2$ in \eqref{h_new}.
We note that these marking strategies
are very similar to the greedy strategies
on a network described in~\cite{DomDSLM17},
where those pipes are marked for a
spatial, temporal, or model refinement
which yield the largest error reduction.

\subsection{The Algorithm}
\label{sec:algorithm-subsec}

With these preliminaries we can now state the overall adaptive model
and discretization control algorithm for finding an
$\eps$-feasible solution of the reference
problem~\eqref{eq:entire-model-infinite-dim}.
The formal listing is given in Alg.~\ref{alg:algorithm}.
\begin{algorithm}
  \label{alg:algorithm}
  \DontPrintSemicolon
  \caption{Adaptive Model and Discretization Control}
  \KwIn{A full specification of the gas network~$\graph = (\nodes,
    \arcs)$, a tolerance~$\eps > 0$,
    initial marking strategy parameters~$\Theta_\td^0, \Theta_\tm^0,
    \Phi_\td^0, \Phi_\tm^0 \in [0,1]$, $\tau^0 \geq 1$,
    and an initial safeguard parameter~$\mu^0 \in \mathbb{N}$.}
  \KwOut{An $\eps$-feasible solution of the reference
    problem~\eqref{eq:entire-model-infinite-dim}.}
  Choose an initial model level~$\ell_\arc^0$ and a
  stepsize~$h_\arc^0$ for every $\arc \in \arcsPipes$.\;
  \label{alg:algorithm:instantiation}
  Solve Problem~\eqref{eq:entire-model-discr} and
  let $\sol^0$ denote the
  optimal solution.\;\label{alg:solveNLP1}
  Compute $\eta_\arc(\sol^0)$ for every $\arc \in
  \arcsPipes$.\;\label{alg:errorsPipes1}
  \If{$\sol^{0}$ is $\eps$-feasible}
  {\label{alg:eps_feas_sol1}
    \Return $\sol^{0}$.
  }
  Set $k = 1$ and
  $\Theta_\td^k  = \Theta_\td^0,
  \Theta_\tm^k  = \Theta_\tm^0,
  \Phi_\td^k  = \Phi_\td^0,
  \Phi_\tm^k  = \Phi_\tm^0,
  \mu^k = \mu^0,
  \tau^k = \tau^0$.\;
  \For{$k = 1, 2, \dotsc$}
  {\label{alg:while-loop}
    \For{$j = 1, \dotsc, \mu^k$}
    {
      Compute the sets~$\switchUpArcs_{k,j}, \refineArcs_{k,j}
      \subseteq \arcsPipes$ according to
      \eqref{eq:marking-strategy-refine} and
      \eqref{eq:marking-strategy-up}.
      \;\label{alg:algorithm:compute-sets-1}
      Switch up the model level for every pipe~$\arc \in
      \switchUpArcs_{k,j}$.
      \;\label{alg:algorithm:model-up}
      Refine the discretization grid for every pipe~$\arc \in
      \refineArcs_{k,j}$.
      \;\label{alg:algorithm:refine-grid}
      Solve Problem~\eqref{eq:entire-model-discr} and
      let $\sol^{k,j}$ denote the
      solution.\;\label{alg:solveNLP2}
      Compute $\eta_{\arc}(\sol^{k,j})$ for every $\arc \in \arcsPipes$.
      \;\label{alg:errorsPipes2}
      \If{$\sol^{k,j}$ is $\eps$-feasible}
      {\label{alg:eps_feas_sol2}
        \Return $\sol^{k,j}$.
      }
    }
    Compute the sets~$\switchDownArcs_k, \coarsenArcs_k \subseteq
    \arcsPipes$ according to
    \eqref{eq:marking-strategy-coarsen} and
    \eqref{eq:marking-strategy-down}.
    \;\label{alg:algorithm:compute-sets-2}
    Switch down the model level for every pipe~$\arc \in
    \switchDownArcs_k$.
    \;\label{alg:algorithm:model-down}
    Coarsen the discretization grid for every pipe~$\arc \in
    \coarsenArcs_k$.
    \;\label{alg:algorithm:coarsen-grid}
    Increase $k \leftarrow k+1$ and update parameters
    $\Theta_\td^{k},
    \Theta_\tm^{k},
    \Phi_\td^{k},
    \Phi_\tm^{k},
    \mu^{k},
    \tau^{k}$.
  }
\end{algorithm}

The algorithm makes use of the safeguard parameter~$\mu \in \mathbb{N}$.
This parameter ensures that the algorithm performs grid coarsenings
and switches down the model level only after applying
$\mu$ rounds of grid refinements and switching up model levels.
It prevents an alternating switching up and down model levels
or an alternating refining and coarsening of the discretization grid.
We note that this technique is similar to the use of hysteresis
  parameters; \cf, \eg \cite{MorMG92}.
By employing this safeguard, we can prove that
Alg.~\ref{alg:algorithm} terminates after a
finite number of iterations with an $\eps$-feasible point of the
reference model~\eqref{eq:stationary-momentum-equ}.

To improve readability, we split the proof of our main theorem into
two parts.
The first lemma states finite termination at an $\eps$-feasible
point if only discretization grid refinements and coarsenings are
applied, whereas the second lemma considers the case of switching
levels in the model hierarchy only, \ie with a fixed stepsize for every
pipe.

\begin{lemma}
  \label{lem:discr-lemma}
  Suppose that the model level~$\ell_\arc \in \set{1,2,3}$ is fixed
  for every pipe~$\arc \in \arcsPipes$.
  Let the resulting set of model levels be denoted by~$\mathcal{M}$.
  Suppose further that $\eta_\arc(\sol) = \eta_{\td,\arc}(\sol)$ holds
  in~\eqref{eq:firstOrderApprox}
  and that every NLP is solved
  to local optimality.
  Consider Alg.~\ref{alg:algorithm} without applying the
  model switching steps in Lines~\ref{alg:algorithm:model-up} and
  \ref{alg:algorithm:model-down}.
  Then, the algorithm terminates after a finite number of refinements
  in Line~\ref{alg:algorithm:refine-grid} and coarsenings in
  Line~\ref{alg:algorithm:coarsen-grid} with an $\eps$-feasible
  solution \wrt\ model level set~$\mathcal{M}$
  if there exists a constant $C > 0$ such that
  \begin{equation}
    \label{eq:discr-parameter-rule}
    \frac{1}{2} \Theta_\td^k \mu^k > \Phi_\td^k + C
  \end{equation}
  holds for all $k$.
\end{lemma}
\begin{proof}
  We consider the total
  discretization error
  \begin{equation*}
    \eta_\td(\sol) = \sum_{\arc \in \arcsPipes} \eta_{\td, \arc}(\sol)
  \end{equation*}
  and show that for every iteration~$k$ 
  the difference between the decrease obtained in the
  inner for-loop 
  and the increase obtained due to
  the coarsenings applied in Line~\ref{alg:algorithm:coarsen-grid}
  is positive and uniformly bounded away from zero.
  In what follows, we only consider a single iteration and drop its
  index~$k$ for better readability.

  First, we consider one refinement step in
  Line~\ref{alg:algorithm:refine-grid}.
  Let $\eta_{\td,\arc}^{j-1}$ denote the discretization error before
  the $j$th inner iteration and let $\eta_{\td,\arc}^{j}$ denote the
  discretization error after the $j$th inner iteration.
  With this, we have
  \begin{align*}
    & \sum_{\arc \in \arcsPipes} \eta_{\td,\arc}^{j-1}
      - \sum_{\arc \in \arcsPipes} \eta_{\td,\arc}^{j}
    \\
    = &
        \sum_{\arc \in \refineArcs_j} \eta_{\td,\arc}^{j-1}
        + \sum_{\arc \in \arcsPipes \setminus \refineArcs_j} \eta_{\td,\arc}^{j-1}
        - \sum_{\arc \in \refineArcs_j} \eta_{\td,\arc}^{j}
        - \sum_{\arc \in \arcsPipes \setminus \refineArcs_j}
        \eta_{\td,\arc}^{j}
    \\
    = &
        \sum_{\arc \in \refineArcs_j} \eta_{\td,\arc}^{j-1}
        - \sum_{\arc \in \refineArcs_j} \eta_{\td,\arc}^{j}
    \\
    = &
        \sum_{\arc \in \refineArcs_j} \frac{1}{2} \eta_{\td,\arc}^{j-1}
  \end{align*}
  for every $j = 1, \dotsc, \mu$.
  For the last equality we have used that
  the implicit Euler method has convergence order~1,
  which (for small stepsizes $h_\arc$)
  implies $\eta_{\td,\arc}^j = \frac{1}{2} \eta_{\td,\arc}^{j-1}$
  when we take the new stepsize as half the current stepsize;
  see~\eqref{new_discr_error}.
  Summing up over all~$\mu$ inner iterations we obtain a telescopic
  sum and finally get an error decrease of
  \begin{equation*}
    \sum_{j=1}^\mu \Bigg( \sum_{\arc \in \arcsPipes}
      \eta_{\td,\arc}^{j-1} - \sum_{\arc \in \arcsPipes}
      \eta_{\td,\arc}^{j} \Bigg)
    =
    \sum_{\arc \in \arcsPipes} \eta_{\td,\arc}^{0}
    - \sum_{\arc \in \arcsPipes} \eta_{\td,\arc}^{\mu}
    =
    \frac{1}{2} \sum_{j=1}^\mu \sum_{\arc \in \refineArcs_j}
    \eta_{\td,\arc}^{j-1}.
  \end{equation*}

  We now consider the coarsening step.
  For this, let $\eta_{\td,\arc}^{\mu}$ denote the discretization error
  before and $\eta_{\td,\arc}^{\mu + 1}$ the discretization error after
  the coarsening step in Line~\ref{alg:algorithm:coarsen-grid}.
  Using similar ideas like above we obtain
  \begin{align*}
    & \sum_{\arc \in \arcsPipes} \eta_{\td,\arc}^{\mu + 1}
      - \sum_{\arc \in \arcsPipes} \eta_{\td,\arc}^{\mu}
    \\
    = \ &
          \sum_{\arc \in \arcsPipes \setminus \coarsenArcs} \eta_{\td,\arc}^{\mu + 1}
          + \sum_{\arc \in \coarsenArcs} \eta_{\td,\arc}^{\mu + 1}
          - \sum_{\arc \in \arcsPipes \setminus \coarsenArcs} \eta_{\td,\arc}^{\mu}
          - \sum_{\arc \in \coarsenArcs} \eta_{\td,\arc}^{\mu}
    \\
    = \ &
          \sum_{\arc \in \coarsenArcs} \eta_{\td,\arc}^{\mu + 1}
          - \sum_{\arc \in \coarsenArcs} \eta_{\td,\arc}^{\mu}
    \\
    = \ &
          2 \sum_{\arc \in \coarsenArcs} \eta_{\td,\arc}^{\mu}
          - \sum_{\arc \in \coarsenArcs} \eta_{\td,\arc}^{\mu}
    \\
    = \ &
          \sum_{\arc \in \coarsenArcs} \eta_{\td,\arc}^{\mu}.
  \end{align*}
  Thus, we are finished if we prove that
  \begin{equation*}
    \frac{1}{2} \sum_{j=1}^\mu \sum_{\arc \in \refineArcs_j}
    \eta_{\td,\arc}^{j-1}
    -
    \sum_{\arc \in \coarsenArcs} \eta_{\td,\arc}^{\mu}
  \end{equation*}
  is positive and uniformly bounded away from zero.
  Using
  \begin{equation*}
    \eta_{\td,\arc}^{j-1} \geq \eta_{\td,\arc}^{\mu},
    \quad
    \text{for all } j = 1, \dotsc, \mu,
  \end{equation*}
  \eqref{eq:marking-strategy-refine},
  \eqref{eq:marking-strategy-coarsen}, and \eqref{eq:discr-parameter-rule},
  we obtain
  \begin{align*}
    & \frac{1}{2} \sum_{j=1}^\mu \sum_{\arc \in \refineArcs_j}
      \eta_{\td, \arc}^{j-1}
      \geq \frac{1}{2} \Theta_\td \sum_{j=1}^\mu \sum_{\arc \in
      \arcsPipes} \eta_{\td, \arc}^{j-1}
      \geq \frac{1}{2} \Theta_\td \sum_{j=1}^\mu \sum_{\arc \in
      \arcsPipes} \eta_{\td, \arc}^{\mu}
    \\
    = \ & \frac{1}{2} \Theta_\td \mu \sum_{\arc \in \arcsPipes}
          \eta_{\td, \arc}^{\mu}
          > (\Phi_\td + C) \sum_{\arc \in \arcsPipes} \eta_{\td, \arc}^{\mu}
          > \sum_{\arc \in \coarsenArcs} \eta_{\td, \arc}^{\mu}
          + C \card{\arcsPipes} \eps,
  \end{align*}
  which completes the proof.
\end{proof}

Next, we prove an analogous lemma for the case that we fix the
stepsize of every arc~$\arc \in \arcsPipes$ and only allow for model
switching.

\begin{lemma}
  \label{lem:model-lemma}
  Suppose that the discretization stepsize $h_\arc$ is fixed for every
  pipe $\arc \in \arcsPipes$.
  Suppose further that $\eta_\arc(\sol) = \eta_{\tm,\arc}(\sol)$ holds
  in~\eqref{eq:firstOrderApprox}
  and that every NLP is solved to local optimality.
  Consider Alg.~\ref{alg:algorithm} without applying the discretization
  refinements in Line~\ref{alg:algorithm:refine-grid} and the coarsenings in
  Line~\ref{alg:algorithm:coarsen-grid}.
  Then, Algorithm~\ref{alg:algorithm} terminates after a finite number of model switches in
  Lines~\ref{alg:algorithm:model-up}
  and~\ref{alg:algorithm:model-down} with an $\eps$-feasible solution
  \wrt\ the stepsizes~$h_\arc$, $\arc \in \arcsPipes$, if there exists
  a constant $C > 0$ such that
  \begin{equation}
    \label{eq:model-parameter-rule}
    \Theta_\tm^k \mu^k > \tau^k \Phi_\tm^k \card{\arcsPipes} + C
  \end{equation}
  holds for all $k$.
\end{lemma}
\begin{proof}
  We consider the total
  model error
  \begin{equation*}
    \eta_\tm(\sol) = \sum_{\arc \in \arcsPipes} \eta_{\tm, \arc}(\sol)
  \end{equation*}
  and show that the difference between
  the decrease obtained in the inner loop
  and the increase obtained due to switching
  model levels down in Line~\ref{alg:algorithm:model-down}
  is positive and uniformly bounded away from zero for
  every iteration~$k$.
  We again consider only a single iteration and drop the corresponding
  index.

  First, we consider a single step of switching up the model level in
  Line~\ref{alg:algorithm:model-up}.
  Let~$\eta_{\tm,\arc}^{j-1}$ denote the model error before the $j$th
  inner iteration and $\eta_{\tm,\arc}^j$ the model error
  after the $j$th inner iteration.
  We then have
  \begin{align*}
    & \sum_{\arc \in \arcsPipes} \eta_{\tm,\arc}^{j-1}
      - \sum_{\arc \in \arcsPipes} \eta_{\tm,\arc}^j
    \\
    = & \sum_{\arc \in \switchUpArcs_j} \eta_{\tm,\arc}^{j-1}
        + \sum_{\arc \in \arcsPipes \setminus \switchUpArcs_j} \eta_{\tm,\arc}^{j-1}
        - \sum_{\arc \in \switchUpArcs_j} \eta_{\tm,\arc}^j
        - \sum_{\arc \in \arcsPipes \setminus \switchUpArcs_j} \eta_{\tm,\arc}^j \\
    = & \sum_{\arc \in \switchUpArcs_j} \eta_{\tm,\arc}^{j-1}
        - \sum_{\arc \in \switchUpArcs_j} \eta_{\tm,\arc}^j
  \end{align*}
  for every $j = 1, \dotsc, \mu$.
  Summing up over all $j$ yields the overall
  model error decrease
  after $\mu$ for-loop iterations of
  \begin{equation*}
    \sum_{j=1}^\mu \Bigg(
      \sum_{\arc \in \arcsPipes} \eta_{\tm,\arc}^{j-1}
      - \sum_{\arc \in \arcsPipes} \eta_{\tm,\arc}^j
    \Bigg)
    =
    \sum_{\arc \in \arcsPipes} \eta_{\tm,\arc}^{0}
    - \sum_{\arc \in \arcsPipes} \eta_{\tm,\arc}^{\mu}
    =
    \sum_{j=1}^\mu \sum_{\arc \in \switchUpArcs_j}
    (\eta_{\tm,\arc}^{j-1}
    - \eta_{\tm,\arc}^j).
  \end{equation*}

  We now consider the step of switching down the model level
  in Line~\ref{alg:algorithm:model-down}.
  Let~$\eta_{\tm,\arc}^{\mu}$ denote the model error before
  and~$\eta_{\tm,\arc}^{\mu + 1}$ the model error after this step.
  It holds that
  \begin{align*}
    & \sum_{\arc \in \arcsPipes} \eta_{\tm,\arc}^{\mu+1} - \sum_{\arc
      \in \arcsPipes} \eta_{\tm,\arc}^\mu \\
    = & \sum_{\arc \in \switchDownArcs} \eta_{\tm,\arc}^{\mu+1}
        + \sum_{\arc \in \arcsPipes \setminus \switchDownArcs} \eta_{\tm,\arc}^{\mu+1}
        - \sum_{\arc \in \switchDownArcs} \eta_{\tm,\arc}^\mu
        - \sum_{\arc \in \arcsPipes \setminus \switchDownArcs} \eta_{\tm,\arc}^\mu \\
    = & \sum_{\arc \in \switchDownArcs} (\eta_{\tm,\arc}^{\mu+1}
        - \eta_{\tm,\arc}^\mu).
  \end{align*}
  Thus, the proof is finished if we show that
  \begin{equation*}
    \sum_{j=1}^\mu \sum_{\arc \in \switchUpArcs_j}
    (\eta_{\tm,\arc}^{j-1}
    - \eta_{\tm,\arc}^j)
    - \sum_{\arc \in \switchDownArcs} (\eta_{\tm,\arc}^{\mu+1}
    - \eta_{\tm,\arc}^\mu)
  \end{equation*}
  is positive and uniformly bounded away from zero.
  With similar ideas as in the proof of Lemma~\ref{lem:discr-lemma}
  and using \eqref{eq:marking-strategy-up},
  \eqref{eq:marking-strategy-down}, and \eqref{eq:model-parameter-rule},
  we obtain
  \begin{align*}
    & \sum_{j=1}^\mu \sum_{\arc \in \switchUpArcs_j}
      (\eta_{\tm,\arc}^{j-1}
      - \eta_{\tm,\arc}^j)
      \geq \Theta_\tm \sum_{j=1}^\mu
      \sum_{\arc \in \arcsPipes^{>\eps}}
      (\eta_{\tm,\arc}^{j-1}
      - \eta_{\tm,\arc}^j)
      > \Theta_\tm \sum_{j=1}^\mu
      \sum_{\arc \in \arcsPipes^{>\eps}}
      \eps
    \\
    = \
    & \Theta_\tm \mu
      \card{\arcsPipes^{>\eps}}
      \eps
      \geq
      \Theta_\tm \mu
      \eps
      >
      \tau \Phi_\tm \card{\arcsPipes} \eps + C \eps
      \geq
      \Phi_\tm \sum_{\arc \in \arcsPipes^{<\eps}(\tau)}
      \eps \tau + C \eps
    \\
    \geq \
    & \Phi_\tm \sum_{\arc \in \arcsPipes^{<\eps}(\tau)}
      (\eta_{\tm,\arc}^{\mu+1}
      - \eta_{\tm,\arc}^\mu)
      + C \eps
      \geq
      \sum_{\arc \in \switchDownArcs} (\eta_{\tm,\arc}^{\mu+1}
      - \eta_{\tm,\arc}^\mu)
      + C \eps,
  \end{align*}
  where we used that $\card{\arcsPipes^{>\eps}} \geq 1$.
  This completes the proof.
\end{proof}

Let $\eta_{\tm,\arc}^\tnew(\sol)$
denote the new model error estimator
after a grid refinement or coarsening.
In order to prove our main theorem
we need to assume that, for every pipe
$\arc \in \arcsPipes$, the change in
the model error estimator after a grid
refinement or coarsening can be
neglected as compared to $\eta_{\tm,\arc}(\sol)$,
\ie $\abs{\eta_{\tm,\arc}(\sol) - \eta_{\tm,\arc}^\tnew(\sol)}
\ll \eta_{\tm,\arc}(\sol)$,
such that we may write
$\eta_{\tm,\arc}^\tnew(\sol) =
\eta_{\tm,\arc}(\sol)$.
A sufficient condition for this
assumption to hold is given by
$\eta_{\td,\arc}(\sol) \ll \eta_{\tm,\arc}(\sol)$
for every $\arc \in \arcsPipes$.
This condition also implies that
$\eta_{\tm,\arc}(\sol)$ is a first-order
approximation of the exact model
error $e_{\tm,\arc}(\sol)$ and is thus
reliable for small stepsizes~$h_\arc$.

%
\begin{lemma}
  \label{lem:reliableModelErrorEstimator}
  Let the discretization and model error estimator
  $\eta_{\td,\arc}(\sol)$ and $\eta_{\tm,\arc}(\sol)$
  as defined in~\eqref{eq:discretizationErrorEstimator}
  and~\eqref{eq:modelErrorEstimator}
  be given for every $\arc \in \arcsPipes$.
  Let further $e_{\tm,\arc}(\sol)$ be the
  exact error between models~(M$_1$)
  and~(M$_{\la}$) and let
  $\eta_{\tm,\arc}^\tnew(\sol)$
  be the new model error estimator
  after a grid refinement or coarsening.
  Then, the implications
  \begin{enumerate}
   \item $\eta_{\td,\arc}(\sol) \ll \eta_{\tm,\arc}(\sol)
   \implies \eta_{\tm,\arc}(\sol) \doteq e_{\tm,\arc}(\sol)$,
   \label{lem:item:reliableModelError}

   \item $\eta_{\td,\arc}(\sol) \ll \eta_{\tm,\arc}(\sol)
   \implies \eta_{\tm,\arc}^\tnew(\sol) = \eta_{\tm,\arc}(\sol)$
   \label{lem:item:neglectibeChangeModelError}
  \end{enumerate}
  hold for every $\arc \in \arcsPipes$.

\end{lemma}
\begin{proof}
  Let pipe $\arc \in \arcsPipes$ be arbitrary.
  To improve readability,  in the following we drop the dependencies of the exact errors and
  the error estimators on~$\arc$ and~$\sol$.
  Without loss of generality, we consider only
  one arbitrary spatial gridpoint~$x_k$.

  Let us first introduce some notation.
  The exact model error is given by
  $e_\tm(x_k) = \hat{\pressure}(\spatVar_k)
  - \pressure^{\Mla}(\spatVar_k)$ for
  the current model level~$\la$,
  the exact discretization error for
  model~(D$_1$) is given by
  $e_\td^1(\spatVar_k) = \hat{\pressure}(\spatVar_k)
  - \pressure^1(\spatVar_k;2h_\arc)$
  and the exact discretization error for
  model~($\Dla$) is denoted by
  $e_\td^{\la}(\spatVar_k) = \pressure^{\Mla}(\spatVar_k)
  - \pressure^{\la}(\spatVar_k;2h_\arc)$.
  Furthermore, the model error estimator
  is given by
  $\eta_\tm(\spatVar_k) = \pressure^1(\spatVar_k;2h_\arc)
  - \pressure^{\la}(\spatVar_k;2h_\arc)$,
  \cf \eqref{eq:modelErrorEstimator}, and we
  define the discretization error estimators
  $\eta_\td^1(\spatVar_k) \define
  \pressure^1(\spatVar_k;2h_\arc)
  - \pressure^1(\spatVar_k;4h_\arc)$
  and
  $\eta_\td^{\la}(\spatVar_k) \define
  \pressure^{\la}(\spatVar_k;2h_\arc)
  - \pressure^{\la}(\spatVar_k;4h_\arc)$
  as in~\eqref{eq:discretizationErrorEstimator}.
  Then, we have $\eta_\td^1(\spatVar_k)
  \doteq e_\td^1(\spatVar_k)$
  and
  $\eta_\td^{\la}(\spatVar_k)
  \doteq e_\td^{\la}(\spatVar_k)$;
  see \cite[page~420]{StoB80}.
  Further, it holds that
  \begin{equation}
   \label{eq:neglectionDiscrError2}
   \abs{\eta_\td^1(\spatVar_k)}
  \ll \abs{\eta_\tm(\spatVar_k)}
  \iff |\eta_\td^{\la}(\spatVar_k)|
  \ll \abs{\eta_\tm(\spatVar_k)},
  \end{equation}
  because $\eta_\td^1(\spatVar_k)$
  and $\eta_\td^{\la}(\spatVar_k)$
  use the same stepsizes $2h_\arc$
  and $4h_\arc$ to compute the
  discrete pressure distributions.

  We now prove implication
  \eqref{lem:item:reliableModelError}.
  Using the previously defined
  notation it holds that
  \begin{align*}
   e_\tm(\spatVar_k) &
   = \hat{\pressure}(\spatVar_k)
  - \pressure^{\Mla}(\spatVar_k) \\
  & = e_\td^1(\spatVar_k)
      + \pressure^1(\spatVar_k;2h_\arc)
      - e_\td^{\la}(\spatVar_k)
      - \pressure^{\la}(\spatVar_k;2h_\arc) \\
  & \doteq \eta_\td^1(\spatVar_k)
      + \pressure^1(\spatVar_k;2h_\arc)
      - \eta_\td^{\la}(\spatVar_k)
      - \pressure^{\la}(\spatVar_k;2h_\arc) \\
  & = \eta_\td^1(\spatVar_k)
      - \eta_\td^{\la}(\spatVar_k)
      + \eta_\tm(\spatVar_k).
  \end{align*}
  Thus, if $\abs{\eta_\td^1(\spatVar_k)}$ and
  $\abs{\eta_\td^{\la}(\spatVar_k)}$
  may be neglected as compared to
  $\abs{\eta_\tm(\spatVar_k)}$, then
  we have $e_\tm(\spatVar_k) \doteq
  \eta_\tm(\spatVar_k)$, \ie
  \begin{equation*}
   \abs{\eta_\td^1(\spatVar_k)}
  \ll \abs{\eta_\tm(\spatVar_k)}
  \, \wedge \, |\eta_\td^{\la}(\spatVar_k)|
  \ll \abs{\eta_\tm(\spatVar_k)}
  \implies
  e_\tm(\spatVar_k) \doteq
  \eta_\tm(\spatVar_k).
  \end{equation*}
  Considering also the equivalence
  relation~\eqref{eq:neglectionDiscrError2}
  it follows that
  \begin{equation*}
   \abs{\eta_\td^1(\spatVar_k)}
   \ll \abs{\eta_\tm(\spatVar_k)}
   \implies
   e_\tm(\spatVar_k) \doteq
   \eta_\tm(\spatVar_k),
  \end{equation*}
  from which implication~\eqref{lem:item:reliableModelError}
  follows directly.

  Finally, we prove implication
  \eqref{lem:item:neglectibeChangeModelError}.
  We show that this implication holds
  for the case that $\eta_\tm^\tnew(\spatVar_k)$
  is the new model error estimator after
  a grid coarsening.
  The case for a grid refinement can
  be shown analogously.
  It holds that
  \begin{align*}
   \eta_\tm^\tnew(\spatVar_k) & =
   \pressure^1(\spatVar_k;4h_\arc)
    - \pressure^{\la}(\spatVar_k;4h_\arc) \\
   & = - \eta_\td^1(\spatVar_k)
       + \pressure^1(\spatVar_k;2h_\arc)
       + \eta_\td^{\la}(\spatVar_k)
       - \pressure^{\la}(\spatVar_k;2h_\arc) \\
  & = - \eta_\td^1(\spatVar_k)
      + \eta_\td^{\la}(\spatVar_k)
      + \eta_\tm(\spatVar_k).
  \end{align*}
  This yields
  \begin{equation}
   \label{lem:eq:neglectionDiscrError3}
   \abs{\eta_\td^1(\spatVar_k)}
   \ll \abs{\eta_\tm(\spatVar_k)}
   \, \wedge \, |\eta_\td^{\la}(\spatVar_k)|
   \ll \abs{\eta_\tm(\spatVar_k)}
   \implies
   \eta_\tm^\tnew(\spatVar_k) = \eta_\tm(\spatVar_k).
  \end{equation}
  Again, considering \eqref{eq:neglectionDiscrError2}
  and~\eqref{lem:eq:neglectionDiscrError3}
  results in
  \begin{equation*}
   \abs{\eta_\td^1(\spatVar_k)}
   \ll \abs{\eta_\tm(\spatVar_k)}
   \implies
   \eta_\tm^\tnew(\spatVar_k) = \eta_\tm(\spatVar_k),
  \end{equation*}
  from which implication~\eqref{lem:item:neglectibeChangeModelError}
  follows immediately.
\end{proof}

With the three preceding lemmas at hand, we are now ready to state and
prove our main theorem about finite termination of
Alg.~\ref{alg:algorithm}.
\begin{theorem}[Finite termination]
  \label{thm:convergence}
  Suppose that $\eta_{\td,\arc} \ll \eta_{\tm,\arc}$
  for every $\arc \in \arcsPipes$ and that every NLP is solved
  to local optimality.
  Then, Algorithm~\ref{alg:algorithm} terminates after a finite number
  of refinements, coarsenings and model switches in
  Lines~\ref{alg:algorithm:model-up}, \ref{alg:algorithm:refine-grid},
  \ref{alg:algorithm:model-down}, and~\ref{alg:algorithm:coarsen-grid}
  with an $\eps$-feasible solution \wrt\ the reference
  problem~\eqref{eq:entire-model-infinite-dim}
  if there exist constants $C_1,C_2 > 0$ such that
  \begin{equation*}
    \frac{1}{2} \Theta_\td^k \mu^k > \Phi_\td^k + C_1,
    \quad
    \Theta_\tm^k \mu^k > \tau^k \Phi_\tm^k \card{\arcsPipes} + C_2
  \end{equation*}
  hold for all $k$.
\end{theorem}
\begin{proof}
  We consider the total error
  $\sum_{\arc \in \arcsPipes} \eta_\arc$
  and show that the difference between the decrease
  obtained in the inner loop and the increase obtained
  due to switching down the model level and coarsening
  the grid is positive and uniformly bounded away from zero
  for every iteration~$k$.
  Again, we consider only a single iteration and drop the
  corresponding index.
  We first consider Lines~\ref{alg:algorithm:model-up} and
  \ref{alg:algorithm:refine-grid} for fixed~$j$.
  It holds that
  \begin{align*}
    & \sum_{\arc \in \arcsPipes} \eta_\arc^{j-1}
      - \sum_{\arc \in \arcsPipes} \eta_\arc^{j}
    \\
    = & \sum_{\arc \in \arcsPipes} \eta_{\tm,\arc}^{j-1}
        + \sum_{\arc \in \arcsPipes} \eta_{\td,\arc}^{j-1}
        - \sum_{\arc \in \arcsPipes} \eta_{\tm,\arc}^{j}
        - \sum_{\arc \in \arcsPipes} \eta_{\td,\arc}^{j}
    \\
    = & \sum_{\arc \in \arcsPipes \setminus (\switchUpArcs_{j} \cup
        \refineArcs_{j})} \eta_{\tm,\arc}^{j-1}
        - \sum_{\arc \in \arcsPipes \setminus (\switchUpArcs_{j}
        \cup \refineArcs_{j})} \eta_{\tm,\arc}^{j}
    + \sum_{\arc \in \switchUpArcs_{j} \setminus \refineArcs_j} \eta_{\tm,\arc}^{j-1}
      - \sum_{\arc \in \switchUpArcs_{j} \setminus \refineArcs_j}
      \eta_{\tm,\arc}^{j}
    \\
    & + \sum_{\arc \in \refineArcs_{j} \setminus \switchUpArcs_{j}} \eta_{\tm,\arc}^{j-1}
      - \sum_{\arc \in \refineArcs_{j} \setminus \switchUpArcs_{j}} \eta_{\tm,\arc}^{j}
     + \sum_{\arc \in \refineArcs_{j} \cap \, \switchUpArcs_{j}} \eta_{\tm,\arc}^{j-1}
      - \sum_{\arc \in \refineArcs_{j} \cap \, \switchUpArcs_{j}} \eta_{\tm,\arc}^{j} \\
    & + \sum_{\arc \in \arcsPipes \setminus (\switchUpArcs_{j} \cup
      \refineArcs_{j})} \eta_{\td,\arc}^{j-1}
      - \sum_{\arc \in \arcsPipes \setminus (\switchUpArcs_{j} \cup
      \refineArcs_{j})} \eta_{\td,\arc}^{j}
     + \sum_{\arc \in \switchUpArcs_{j} \setminus \refineArcs_j} \eta_{\td,\arc}^{j-1}
      - \sum_{\arc \in \switchUpArcs_{j} \setminus \refineArcs_j}
      \eta_{\td,\arc}^{j}
    \\
    & + \sum_{\arc \in \refineArcs_{j} \setminus \switchUpArcs_j} \eta_{\td,\arc}^{j-1}
      - \sum_{\arc \in \refineArcs_{j} \setminus \switchUpArcs_j} \eta_{\td,\arc}^{j}
     + \sum_{\arc \in \refineArcs_{j} \cap \, \switchUpArcs_{j}} \eta_{\td,\arc}^{j-1}
      - \sum_{\arc \in \refineArcs_{j} \cap \, \switchUpArcs_{j}} \eta_{\td,\arc}^{j} \\
    = & \sum_{\arc \in \switchUpArcs_{j}} \eta_{\tm,\arc}^{j-1}
        - \sum_{\arc \in \switchUpArcs_{j}} \eta_{\tm,\arc}^{j}
        + \sum_{\arc \in \refineArcs_{j}} \eta_{\td,\arc}^{j-1}
        - \sum_{\arc \in \refineArcs_{j}} \eta_{\td,\arc}^{j}
    \\
    = & \sum_{\arc \in \switchUpArcs_{j}} (\eta_{\tm,\arc}^{j-1}
        - \eta_{\tm,\arc}^{j})
        + \frac{1}{2} \sum_{\arc \in \refineArcs_{j}} \eta_{\td,\arc}^{j-1},
  \end{align*}
  where we use that
  $\eta_{\tm,\arc}^{j} = \eta_{\tm,\arc}^{j-1}$
  for every $\arc \in \refineArcs_{j} \setminus \switchUpArcs_{j}$
  since $\eta_{\td,\arc}^{j-1} \ll \eta_{\tm,\arc}^{j-1}$
  for every $\arc \in \arcsPipes$;
  see Lemma~\ref{lem:reliableModelErrorEstimator}.
  Moreover, the discretization error
  estimator~$\eta_{\td,\arc}$
  does not change after a switching up the model level.

  Again, summing up over all $j = 1, \dotsc, \mu$ yields the overall
  error decrease after~$\mu$ for-loop iterations of
  \begin{align*}
    & \sum_{j=1}^\mu \Bigg(\sum_{\arc \in \arcsPipes} \eta_\arc^{j-1}
      - \sum_{\arc \in \arcsPipes} \eta_\arc^{j}\Bigg)
      = \sum_{\arc \in \arcsPipes} \eta_{\arc}^{0}
      - \sum_{\arc \in \arcsPipes} \eta_{\arc}^{\mu} \\
    = & \sum_{j=1}^\mu \Bigg(\sum_{\arc \in \switchUpArcs_j}
        (\eta_{\tm,\arc}^{j-1} - \eta_{\tm,\arc}^j)
        + \frac{1}{2} \sum_{\arc \in \refineArcs_{j}} \eta_{\td,\arc}^{j-1}\Bigg).
  \end{align*}

  With similar arguments as before for Lines~\ref{alg:algorithm:model-up}
  and~\ref{alg:algorithm:refine-grid}
  we consider Lines~\ref{alg:algorithm:model-down}
  and~\ref{alg:algorithm:coarsen-grid} and obtain
  \begin{align*}
    & \sum_{\arc \in \arcsPipes} \eta_\arc^{\mu + 1}
      - \sum_{\arc \in \arcsPipes} \eta_\arc^{\mu}
    \\
    = & \sum_{\arc \in \arcsPipes} \eta_{\td,\arc}^{\mu + 1}
        + \sum_{\arc \in \arcsPipes} \eta_{\tm,\arc}^{\mu + 1}
        - \sum_{\arc \in \arcsPipes} \eta_{\td,\arc}^{\mu}
        - \sum_{\arc \in \arcsPipes} \eta_{\tm,\arc}^{\mu}
    \\
    = & \sum_{\arc \in \coarsenArcs} \eta_{\td,\arc}^{\mu + 1}
        - \sum_{\arc \in \coarsenArcs} \eta_{\td,\arc}^{\mu}
        + \sum_{\arc \in \switchDownArcs} \eta_{\tm,\arc}^{\mu + 1}
        - \sum_{\arc \in \switchDownArcs} \eta_{\tm,\arc}^{\mu}
    \\
    = & \sum_{\arc \in \coarsenArcs} \eta_{\td,\arc}^{\mu}
        + \sum_{\arc \in \switchDownArcs} (\eta_{\tm,\arc}^{\mu + 1}
        -  \eta_{\tm,\arc}^{\mu}).
  \end{align*}
  Finally, it remains to prove that
  \begin{equation*}
    \sum_{j=1}^\mu \Bigg(\sum_{\arc \in \switchUpArcs_j}
      (\eta_{\tm,\arc}^{j-1} - \eta_{\tm,\arc}^j)
      + \frac{1}{2} \sum_{\arc \in \refineArcs_{j}} \eta_{\td,\arc}^{j-1}\Bigg)
    - \sum_{\arc \in \coarsenArcs} \eta_{\td,\arc}^{\mu}
    - \sum_{\arc \in \switchDownArcs} (\eta_{\tm,\arc}^{\mu + 1}
    -  \eta_{\tm,\arc}^{\mu})
  \end{equation*}
  is positive and uniformly bounded away from zero.
  Using the proofs of Lemmas~\ref{lem:discr-lemma} and~\ref{lem:model-lemma}
  we have
  \begin{align*}
    & \sum_{j=1}^\mu \sum_{\arc \in \switchUpArcs_j}
      (\eta_{\tm,\arc}^{j-1} - \eta_{\tm,\arc}^j)
      + \frac{1}{2} \sum_{j=1}^\mu
      \sum_{\arc \in \refineArcs_{j}} \eta_{\td,\arc}^{j-1}
    \\
    > \, & \Theta_\tm \mu \eps
           + \frac{1}{2} \mu \Theta_\td \sum_{\arc \in \arcsPipes}
           \eta_{\td, \arc}^{\mu} \\
    > \, & \tau \Phi_\tm \card{\arcsPipes} \eps + C_2 \eps
           + (\Phi_\td + C_1) \sum_{\arc \in \arcsPipes} \eta_{\td, \arc}^{\mu} \\
    > \, & \sum_{\arc \in \switchDownArcs} (\eta_{\tm,\arc}^{\mu+1}
           - \eta_{\tm,\arc}^\mu) + C_2 \eps
           + \sum_{\arc \in \coarsenArcs} \eta_{\td, \arc}^{\mu}
           + C_1 \card{\arcsPipes} \eps,
  \end{align*}
  which completes the proof.
\end{proof}

\subsection{Remarks}
\label{sec:remarks}

Before we close this section we discuss some details and extensions
regarding Alg.~\ref{alg:algorithm}.
First, we give an overview of the main computations that are
performed in the algorithm.
In Lines~\ref{alg:solveNLP1} and~\ref{alg:solveNLP2}, the
NLP~\eqref{eq:entire-model-discr} is solved using the
current model level~$\la$ and the current stepsize~$h_\arc$ for every
pipe $\arc \in \arcsPipes$.
Most types of NLP algorithms are iterative methods.
That is, the computational costs of the algorithms depend on the
number of iterations required to converge to a (local) optimal
solution and the costs per iteration.
The latter mainly consist of the solution of a linear system (\eg
suitable forms of the KKT system for interior-point or active-set
methods) for computing the search direction.
The size of this linear system typically is $\mathcal{O}(n+m)$, where
$n$ is the number of variables and $m$ is the number of constraints of
the NLP.
Both $n$ and $m$ are directly controlled by the stepsizes~$h_\arc$
that we use in our NLP models.
The model level~$\la$ mainly determines the sparsity/density of the system matrices of the linear
systems and the overall nonlinearity of the NLP, which typically
influences the number of required iterations.

In Lines~\ref{alg:errorsPipes1} and~\ref{alg:errorsPipes2}, the
overall error estimator $\eta_{\arc}(\sol)$
is computed for every pipe $\arc \in \arcsPipes$.
Thus, for all pipes, the solution of model~($\tD_1$) is computed
with stepsize both~$2h_\arc$ and~$4h_\arc$
and the solution of model ($\Dla$) is computed with stepsize~$h_\arc$.
These solutions are obtained by solving the initial value problems
consisting of the ordinary differential
equations~($\tM_1$) and~($\Mla$) together
with the initial value $\pressure(\spatVar_0)$, which
is contained in the optimal solution~$\sol$ of
Problem~\eqref{eq:entire-model-discr}.
Continuing with the example of the implicit Euler method that we
use as numerical integration scheme throughout this paper, the initial value
problems can be solved (i) by considering the implicit equations
in~($\tD_1$) and~($\Dla$) and using, \eg the Newton method to solve
for~$\pressure_k$ in every space integration step or (ii) by using
an existing software code and setting the order of
the numerical integration scheme to one.

The subset~$\mathcal{R}$ in Line~\ref{alg:algorithm:compute-sets-1}
can be determined efficiently, since $\eta_{\td,\arc}(\sol)$
has already been computed in Line~\ref{alg:errorsPipes1}
or~\ref{alg:errorsPipes2} for every $\arc \in \arcsPipes$.
For subset~$\mathcal{U}$ in Line~\ref{alg:algorithm:compute-sets-1}
and in~\eqref{eq:marking-strategy-up}
the error estimator $\eta_{\tm,\arc}(\sol)$
has also already been computed in Line~\ref{alg:errorsPipes1}
or~\ref{alg:errorsPipes2} for every $\arc \in \arcsPipes$.
Moreover, $\la^\tnew$ in~\eqref{level_switch_up_rule}
has to be computed in order to determine~$\mathcal{U}$.
For this, we compute $\eta_{\tm,\arc}(\sol;\la-1)$
if and only if $\la = 3$.
In the case $\la = 2$ we have $\eta_{\tm,\arc}(\sol;\la-1) = 0$
and for $\la = 1$ we have $\eta_{\tm,\arc}(\sol;\la^\tnew)
= \eta_{\tm,\arc}(\sol;\la) = 0$.
Subset~$\mathcal{C}$ in Line~\ref{alg:algorithm:compute-sets-2}
can also be computed efficiently, since $\eta_{\td,\arc}(\sol)$
has already been computed in Line~\ref{alg:errorsPipes1}
or~\ref{alg:errorsPipes2} for every $\arc \in \arcsPipes$.
For subset~$\mathcal{D}$ in Line~\ref{alg:algorithm:compute-sets-2}
and in~\eqref{eq:marking-strategy-down}
the error estimator $\eta_{\tm,\arc}(\sol)$
has been computed already in Line~\ref{alg:errorsPipes1}
or~\ref{alg:errorsPipes2} for every $\arc \in \arcsPipes$.
If $\la \in \set{1,2}$, then $\eta_{\tm,\arc}(\sol;\la+1)$
has to be computed for every $\arc \in \arcsPipes$
in order to determine~$\mathcal{D}$.

We note that the optimal solution~$\sol$
of Problem~\eqref{eq:entire-model-discr} contains,
among others, the model level~$\la$, stepsize~$h_\arc$, and
pressure~$\pressure^{\la}(\spatVar_0)$ at the
beginning of the pipe, for every $\arc \in \arcsPipes$.
Using $\la$, $h_\arc$, and $\pressure^{\la}(\spatVar_0)$,
the discretization and model error estimator for
pipe $\arc \in \arcsPipes$ can be computed
without information from other pipes.
Hence, the error estimators, \eg in
Line~\ref{alg:errorsPipes2}, can be computed in parallel.

Up to now, we have discussed two types of errors:
modeling and discretization errors.
Both are handled by Alg.~\ref{alg:algorithm} and we have shown that
the algorithm terminates with a combined model and discretization
error that satisfies a user-specified error tolerance~$\eps > 0$.
What we have ignored so far is that the NLPs are also solved by a
numerical method that introduces numerical errors as well.
However, it is easy to integrate the control of this additional error
source into Alg.~\ref{alg:algorithm}.
Let $\eps^\topt > 0$ be the optimality tolerance that we hand
over to the optimization solver and suppose that the solver always
satisfies this tolerance.
Furthermore, let the tolerance~$\eps$ considered so far now be denoted
by $\eps^\tdm$.
Using the triangle inequality we easily see that the upper bound of
the total error (that is aggregated modeling, discretization, and
optimization error) is $\eps^\topt + \eps^\tdm$.
Hence, in order to satisfy an overall error tolerance~$\eps > 0$,
we have to ensure that $\eps^\topt + \eps^\tdm \leq \eps$ holds,
which can be formally introduced in~Alg.~\ref{alg:algorithm} by
replacing $\eps$ with $\eps^\topt + \eps^\tdm$.

Finally, note that this additional error source directly suggests
itself for adaptive treatment as well.
In the early iterations of Alg.~\ref{alg:algorithm} it is not
important that $\eps^\topt$ is small.
That is, the optimization is allowed to produce coarser approximate local
solutions.
However, in the course of the algorithm, one can observe the achieved
modeling and discretization error and can adaptively tighten the
optimization tolerance.
Since this strategy allows the optimization method to produce coarse
approximate solutions in the beginning, it can be expected that this
leads to a speed-up in the overall running times of
Alg.~\ref{alg:algorithm}.

The choice of the error tolerance $\eps$
that has to be provided in Alg.~\ref{alg:algorithm}
will depend on the user requirements,
however, one should be aware that
due to the round-off errors committed during every single step of the procedure,
and due to possible ill-conditioning of the linear systems solved by the NLP solver,
none of the three errors,
the discretization error, the modeling error, and the NLP error
can be chosen extremely small.
Since the backward error and the associated condition number of the linear systems
can be estimated during the procedure, see \cite{GolV96},
and since the error estimates for the discretization method are at hand,
it is just the modeling error which is not known a priori.
To estimate this latter error (of the finest model) usually requires a comparison with experimental data.
If these are available during a real-world process,
then it is possible to adjust the required tolerances~$\eps$
in a feedback loop using a standard PI controller, see, \eg \cite{PetCK91},
i.e., if measured data are available that show that the finest model has a given accuracy,
then~$\eps$ should not be chosen smaller than this.

Finally, we want to stress that the described adaptive error control algorithm
can be used with any number of model levels in the hierarchy,
with any higher order discretization scheme,
and with any number of grid refinement levels.


\section{Computational Results}
\label{sec:comp-results}

In this section we present numerical results obtained by the adaptive
error control algorithm.
To this end, we compare the efficiency of the method with an approach
that directly solves an NLP that satisfies the same error tolerance
and that is obtained without using adaptivity.
Before we discuss the results in detail we brief\/ly mention the
computational setup and the gas transport network instances that we
solve.

We implemented the adaptive error control
algorithm~\ref{alg:algorithm} in \codeName{Python~2.7.13} and used the
\codeName{scipy~0.14.0} module for solving the initial value problems.
All nonlinear optimization models have been implemented using the
\codeName{C++} framework
\codeName{LaMaTTO++}\footnote{\url{http://www.mso.math.fau.de/edom/projects/lamatto.html}}
for modeling and solving mixed-integer nonlinear optimization problems on
networks.
The computations have been done on a six-core AMD Opteron\textsuperscript{TM}
Processor 2435 with \SI{2.2}{\giga\hertz} and \SI{64}{\giga\byte} of
main memory.
The NLPs have been solved using \Ipopt~3.12; see
\cite{waechter_biegler:2005,waechter_biegler:2006}.

For our computational study, we choose publicly available
\codeName{GasLib} instances; see~\cite{Schmidt_et_al:2017b}.
This has the advantage that, if desired, all numerical results can be
reproduced on the same data.
In what follows, we consider the networks \codeName{GasLib-40} and
\codeName{GasLib-135}, since these are the largest networks in the
\codeName{GasLib} that only contain pipes and compressor stations as
arc types.
Detailed statistics are given in Table~\ref{tab:network-stats}.
\begin{table}
  \centering
  \caption{Statistics for the instances}
  \label{tab:network-stats}
  \begin{tabular}{ccccc}
    \toprule
    Network & \# nodes & \# pipes & \# compressor stations & total
                                                             pipe
                                                             length
                                                             (km) \\
    \midrule
    \codeName{GasLib-40} & 40 & 39 & 6 & \num[round-mode=places,round-precision=0]{1112.4705743774798}\\
    \codeName{GasLib-135} & 135 & 141 & 29 & \num[round-mode=places,round-precision=0]{6934.585662953983}\\
    \bottomrule
  \end{tabular}
\end{table}

Next, we describe the parameterization of
Alg.~\ref{alg:algorithm}.
We initialize every pipe $\arc \in \arcsPipes$ with the coarsest model
level~$\ell_\arc = 3$ and with the coarsest possible discretization
grid.
In order to yield a well-defined algorithm, the number of
discretization grid intervals has to be a multiple of $4$; \cf
Fig.~\ref{fig:evaluationGrid}.
Thus, we initially set $h_\arc = \length_\arc / 4$
and ensure in Step~\ref{alg:algorithm:coarsen-grid} of
Alg.~\ref{alg:algorithm} that we never obtain a coarser grid size than
the initial one.
The overall tolerance is set to $\eps = 10^{-4}$\,\si{\bar}.
Moreover, we set $\Theta_\td = \Theta_\tm = 0.7$, $\Phi_\td = \Phi_\tm = 0.3$,
$\tau = 1.1$, and $\mu = 4$.
Here, we refrain from updating these parameters from iteration to
iteration, which is possible in general.
Note that our parameter choice violates the second inequality of
Theorem~\ref{thm:convergence}.
This could be fixed by simply increasing the hysteresis
parameter~$\mu$.
However, we refrain from using a larger $\mu$ in order to give the
adaptive algorithm more chances to also switch down in the model
hierarchy or to coarsen discretization grids.
Our numerical experiments show that the violation of the second
inequality of Theorem~\ref{thm:convergence} does not harm convergence
in practice but leads to slightly faster computations.

The same rationale holds for the relation between model and
discretization error as assumed in Theorem~\ref{thm:convergence};
\cf also Lemma~\ref{lem:reliableModelErrorEstimator}.
To be fully compliant with the theory, the initial discretization grids
need to be much finer.
Again, coarser initial discretization grids do not harm convergence
in our numerical experiments but yield much faster computations.

We now turn to the discussion of the numerical results.
Both instances are solved using 8~iterations.
Thus, together with the initially solved NLP,
we have to solve 9~NLPs for solving both instances.

Using the adaptive control algorithm,
it takes
\SI[round-mode=places,round-precision=2]{3.82332675552}{\second} to
solve the \codeName{GasLib-40} instance and
\SI[round-mode=places,round-precision=2]{7.49602046585}{\second} to
solve the \codeName{GasLib-135} instance.
For the \codeName{GasLib-40} network, the final NLP contains
\num{2026}~variables and \num{1988}~constraints, whereas for the
\codeName{GasLib-135} the final NLP contains \num{3405}~variables and
\num{3271}~constraints.

Most interesting is the speed-up that we obtain
by using the adaptive control algorithm.
Thus, we compare the above given solution times with the solution times
for an NLP that satisfies the same error tolerances but that is
obtained without using model level and discretization grid adaptivity.
This NLP contains \num{40034}~variables and \num{39996}~constraints
for the \codeName{GasLib-40} instance and \num{144757}~variables as
well as \num{144623}~constraints for the \codeName{GasLib-135}
instance.
Compared to the final NLPs that have to be solved within the adaptive
algorithm, the NLPs obtained without using adaptivity are quite large
scale.
This directly translates to solution times.
The \codeName{GasLib-40} instance requires \SI{53.11}{\second} and
the \codeName{GasLib-135} instance requires \SI{122.42}{\second}.
Thus, we get a speed-up factor of 13.89 and 16.33, respectively.

Figure~\ref{fig:RU-plot} illustrates the adaptivity of the algorithm by
plotting how many pipe grids are refined ($\card{\mathcal{R}}$) and
how many pipe models are switched up in the hierarchy
($\card{\mathcal{U}}$).
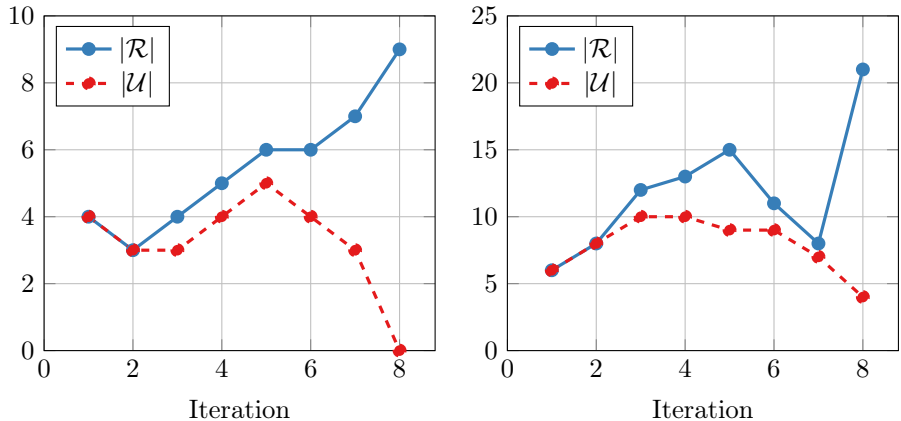
\begin{figure}
  \centering
  \begin{tikzpicture}
  \begin{axis}[%
  scale only axis,
  scale=0.61,
  ymin=0, ymax=10,
  xmin=0,
  xlabel={Iteration},
  legend pos=north west,
  grid]
    \addplot[solid, very thick, mark=*, color=mycol1] table [x=iter, y=R]{tikz-imgs/RU-plot-gaslib-40.dat};
    \addlegendentry{$\card{\mathcal{R}}$}
    \addplot[dashed, very thick, mark=*, color=mycol2] table [x=iter, y=U]{tikz-imgs/RU-plot-gaslib-40.dat};
    \addlegendentry{$\card{\mathcal{U}}$}
  \end{axis}
\end{tikzpicture}%
%
%
\quad%
  \begin{tikzpicture}
  \begin{axis}[%
  scale only axis,
  scale=0.61,
  ymin=0, ymax=25,
  ytick={0,5,10,15,20,25},
  xmin=0,
  xlabel={Iteration},
  legend pos=north west,
  grid]
    \addplot[solid, very thick, mark=*, color=mycol1] table [x=iter, y=R]{tikz-imgs/RU-plot-gaslib-135.dat};
    \addlegendentry{$\card{\mathcal{R}}$}
    \addplot[dashed, very thick, mark=*, color=mycol2] table [x=iter, y=U]{tikz-imgs/RU-plot-gaslib-135.dat};
    \addlegendentry{$\card{\mathcal{U}}$}
  \end{axis}
\end{tikzpicture}%
%
%
  \caption{Number of pipes with refined grid ($y$-axis; $\card{\mathcal{R}}$)
    and number of pipes where the model is switched up in the model hierarchy ($y$-axis;
    $\card{\mathcal{U}}$) over the course of the iterations ($x$-axis).
    Left: \codeName{GasLib-40}, right: \codeName{GasLib-135}.}
  \label{fig:RU-plot}
\end{figure}
It can be clearly seen that increasing the accuracy is only needed for
a small fraction of the pipes.
For the \codeName{GasLib-40} network, we never refine grids for more
than 9~pipes, whereas we never refine grids for more than 21 pipes for
the \codeName{GasLib-135} network.
Thus, for the larger network, we never refine grids for more than
\SI{15}{\percent} of all pipes.

For both networks, the Lines~\ref{alg:algorithm:model-down} and
\ref{alg:algorithm:coarsen-grid} are only reached once.
For the smaller network, only 1 pipe grid is coarsened, whereas 3 pipe grids
are coarsened for the larger network.
Moreover, the algorithm never switches down in the model hierarchy.
Consequently, the NLPs get larger from iteration to iteration.
This then yields increased running times for the NLP solver as
depicted in Fig.~\ref{fig:sol-times}.
\begin{figure}
  \centering
  \begin{tikzpicture}
  \begin{axis}[%
    scale only axis,
    scale=0.56,
    legend pos=north west,
    ymode=log,
    ymax=10,
    ylabel={CPU time [s]},
    ylabel shift={-0.2cm},
    xlabel={Iteration},
    xmin=0,
    grid]
    \addplot[dotted, very thick, mark=*, color=mycol1] table [x=iter, y=t1]{tikz-imgs/sol-times-gaslib-40.dat};
    \addlegendentry{NLPs}
    \addplot[dashed, very thick, mark=*, color=mycol2] table [x=iter, y=t2]{tikz-imgs/sol-times-gaslib-40.dat};
    \addlegendentry{IVPs}
  \end{axis}
\end{tikzpicture}%
%
%
\quad%
  \begin{tikzpicture}
  \begin{axis}[%
    scale only axis,
    scale=0.56,
    legend pos=north west,
    ymode=log,
    ymin=1e-1, ymax=1e1,
    ytick={1e-1,1,1e1},
    xlabel={Iteration},
    xmin=0,
    grid]
    \addplot[dotted, very thick, mark=*, color=mycol1] table [x=iter, y=t1]{tikz-imgs/sol-times-gaslib-135.dat};
    \addlegendentry{NLPs}
    \addplot[dashed, very thick, mark=*, color=mycol2] table [x=iter, y=t2]{tikz-imgs/sol-times-gaslib-135.dat};
    \addlegendentry{IVPs}
  \end{axis}
\end{tikzpicture}%
%
%
  \caption{Aggregated run times ($y$-axis; in s) required for solving the
    nonlinear optimization problems (NLPs) and the initial value problems (IVPs)
    for the computation of the error estimates.
    Left: \codeName{GasLib-40}, right: \codeName{GasLib-135}.}
  \label{fig:sol-times}
\end{figure}
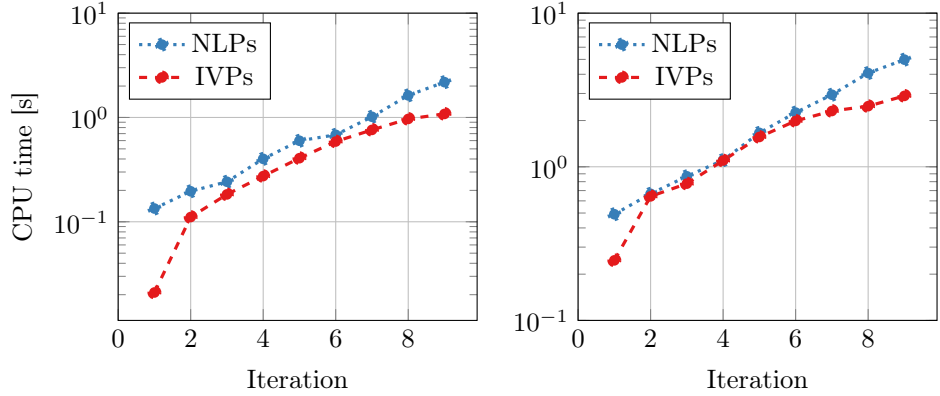
It can be seen that the subsequent NLPs can be solved quite fast.
There are two main reasons for this phenomenon.
First, the NLP's size only increases moderately due to the adaptive
control strategy.
Second, the overall algorithm allows for warm-starting:
When solving a single NLP we always use the last NLP's solution to
set up the initial iterate.

Lastly, we consider the decrease in the respective errors.
In Fig.~\ref{fig:errors}, the discretization, model, and total errors
are plotted over the course of the iterations.
Both profiles show the expected decrease in the errors.
\begin{figure}
  \centering
  \begin{tikzpicture}
  \begin{axis}[%
    scale only axis,
    scale=0.56,
    ymode=log,
    ymin=1e-5, 
    ylabel={Error estimate},
    xlabel={Iteration},
    xmin=0,
    grid]
    \addplot[dotted, very thick, mark=*, color=mycol1] table [x=iter, y expr=\thisrow{eta_d}*1e-5]{tikz-imgs/errors-gaslib-40.dat};
    \addlegendentry{$\eta_\td$}
    \addplot[dashed, very thick, mark=*, color=mycol2] table [x=iter, y expr=\thisrow{eta_m}*1e-5]{tikz-imgs/errors-gaslib-40.dat};
    \addlegendentry{$\eta_\tm$}
    \addplot[solid, very thick, mark=*, color=mycol3] table [x=iter, y expr=\thisrow{eta}*1e-5]{tikz-imgs/errors-gaslib-40.dat};
    \addlegendentry{$\eta$}
  \end{axis}
\end{tikzpicture}%
%
%
\quad%
  \begin{tikzpicture}
  \begin{axis}[%
    scale only axis,
    scale=0.56,
    ymode=log,
    ymin=1e-5, 
    xlabel={Iteration},
    xmin=0,
    grid]
    \addplot[dotted, very thick, mark=*, color=mycol1] table [x=iter, y expr=\thisrow{eta_d}*1e-5]{tikz-imgs/errors-gaslib-135.dat};
    \addlegendentry{$\eta_\td$}
    \addplot[dashed, very thick, mark=*, color=mycol2] table [x=iter, y expr=\thisrow{eta_m}*1e-5]{tikz-imgs/errors-gaslib-135.dat};
    \addlegendentry{$\eta_\tm$}
    \addplot[solid, very thick, mark=*, color=mycol3] table [x=iter, y expr=\thisrow{eta}*1e-5]{tikz-imgs/errors-gaslib-135.dat};
    \addlegendentry{$\eta$}
  \end{axis}
\end{tikzpicture}%
%
%
  \caption{Discretization, model and total error estimates ($y$-axis) over the course of
    the iterations ($x$-axis). Left: \codeName{GasLib-40}, right:
    \codeName{GasLib-135}.}
  \label{fig:errors}
\end{figure}
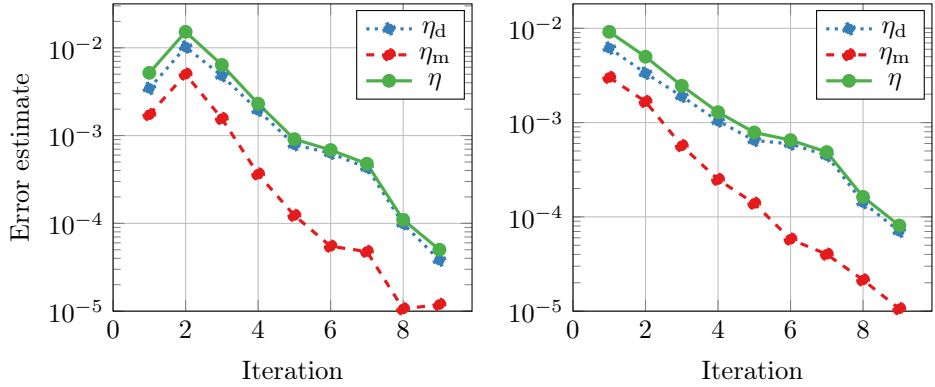


\section{Conclusion}
\label{sec:conclusion}

We have considered the problem of operation cost
minimization for gas transport networks.
In this context, we have focused on stationary and isothermal models
and developed an adaptive model and discretization error control
algorithm for nonlinear optimization that uses a hierarchy of
continuous and finite-dimensional models.
Out of this hierarchy, the new method
adaptively chooses different models in order to finally achieve an
optimal solution that satisfies a prescribed combined model and
discretization error tolerance.
The algorithm is shown to be convergent and its performance is
illustrated by several numerical results.

The results pave the way for future work in the context of
model switching and discretization grid adaptation for nonlinear optimal
control. On the one hand, it should be extended to non-isothermal
and instationary models of gas transport, in particular, in a
port-Hamiltonian formulation. On the other hand, it would be
interesting to extend the new technique to mixed-integer nonlinear
optimal control.


\section*{Acknowledgements}
\label{sec:acknowledgements}

This research has been performed as part of the Energie Campus
Nürnberg and is supported by funding of the Bavarian State Government.
The authors acknowledge funding through the DFG Transregio TRR~154,
subprojects B03 and B08.


\bibliographystyle{plain}
\bibliography{nlp-model-coupling}

\end{document}